\documentclass[11pt]{amsart}
\usepackage[top=1.25in,bottom=1.25in,left=1in,right=1in]{geometry}
\usepackage{amsfonts, amsmath, amssymb, amscd, amsthm, graphicx}
\usepackage[final, 
colorlinks=true,
pdftitle={Right-angled Artin groups as normal subgroups of Mapping Class Groups},
pdfauthor={Matt Clay, Dan Margalit, Johanna Mangahas}]{hyperref}
\usepackage{enumitem,diagbox,multirow}

\usepackage{tikz}
\usetikzlibrary{calc}
\makeatletter
\newcommand{\gettikzxy}[3]{%
  \tikz@s\can@one@point\pgfutil@firstofone#1\relax
  \edef#2{\the\pgf@x}%
  \edef#3{\the\pgf@y}%
}
\makeatother

\colorlet{mc0}{black}
\colorlet{mc1}{blue}
\colorlet{mc2}{red}
\colorlet{mc3}{black!50!green}

\newcommand{\nc}{\newcommand}
\nc{\nt}{\newtheorem}
\nc{\dmo}{\DeclareMathOperator}

\nt{theorem}{Theorem}
\nt{lemma}[theorem]{Lemma}
\nt{proposition}[theorem]{Proposition}
\nt{corollary}[theorem]{Corollary}
\nt{conjecture}[theorem]{Conjecture}
\nt{claim}[theorem]{Claim}
\nt{fact}[theorem]{Fact}
\nt*{claim*}{Claim}
\theoremstyle{definition}
\nt{definition}[theorem]{Definition}
\nt{remark}[theorem]{Remark}
\nt{problem}[theorem]{Problem}
\nt{question}[theorem]{Question}

\numberwithin{theorem}{section}

\dmo{\Mod}{Mod}
\dmo{\Teich}{Teich}
\dmo{\Aut}{Aut}
\dmo{\Out}{Out}
\dmo{\Homeo}{Homeo}
\dmo{\Stab}{Stab}
\dmo{\Star}{Star}
\dmo{\Link}{Link}
\dmo{\PMF}{PMF}
\dmo{\EC}{EC}

\nc{\A}{\mathcal{A}}
\nc{\calB}{\mathcal{B}}
\nc{\C}{\mathcal{C}}
\nc{\F}{\mathcal{F}}
\nc{\calG}{\mathcal{G}}
\nc{\calH}{\mathcal{H}}
\nc{\calL}{\mathcal{L}}
\nc{\calN}{\mathcal{N}}
\nc{\calO}{\mathcal{O}}
\nc{\X}{\mathcal{X}}
\nc{\calY}{\mathcal{Y}}
\nc{\calZ}{\mathcal{Z}}
\nc{\Q}{\mathcal{Q}}
\nc{\T}{\mathcal{T}}
\nc{\W}{\mathrm{Piv}}
\renewcommand{\P}{\mathcal{P}}
\nc{\I}{\mathcal{I}}
\nc{\PB}{PB}
\nc{\D}{D}

\nc{\Y}{\mathbb{Y}}

\nc{\N}{\mathbb{N}}
\nc{\Z}{\mathbb{Z}}
\nc{\R}{\mathbb{R}}

\nc{\Ke}{C_e}
\nc{\Kp}{C_p}
\nc{\Kg}{C_g}

\nc{\grp}[1]{\langle{#1}\rangle}
\nc{\normalclosure}[1]{\left\langle\hspace{-.7mm}\left\langle{#1}\right\rangle\hspace{-.7mm}\right\rangle}%
\nc{\abs}[1]{\left\lvert {#1} \right\rvert}%
\nc{\cut}[1]{\left\{\!\left\{ {#1} \right\}\!\right\}}%

\nc{\from}{\colon\thinspace}
\nc{\dd}{\underline{d}}
\nc{\BigFreeProd}[1]{\raisebox{-1.5pt}{ \ensuremath{\underset{\mbox{\scriptsize{$#1$}}}{\mbox{\Huge{$\ast$}}}}}\,}
\nc{\FreeProd}[1]{\raisebox{-1.5pt}{ \ensuremath{\underset{\mbox{\scriptsize{$#1$}}}{\mbox{\huge{$\ast$}}}}}\,}
\nc{\FreeProdNoArg}{\raisebox{-1.5pt}{ {\mbox{\huge{$\ast$}}}}\,}
\nc{\FreeProdInf}{\raisebox{-1.5pt}{ \ensuremath{\underset{\mbox{\scriptsize{$\infty$}}}{\mbox{\huge{$\ast$}}}}}\,}
\nc{\margin}[1]{\marginpar{\scriptsize #1}}
\nc{\p}[1]{\medskip\paragraph{{\it #1}}}
\nc{\ip}[1]{\medskip\paragraph{{\em #1}}}

\nc{\param}
	{{\mathchoice{\mkern1mu\mbox{\raise2.2pt\hbox{$\centerdot$}}\mkern1mu}%
	{\mkern1mu\mbox{\raise2.2pt\hbox{$\centerdot$}}\mkern1mu}%
	{\mkern1.5mu\centerdot\mkern1.5mu}{\mkern1.5mu\centerdot\mkern1.5mu}}}

\dmo{\diam}{diam}
\dmo{\interior}{int}
\dmo{\supp}{supp}
\dmo{\PSL}{PSL}
\dmo{\Aff}{Aff}
\dmo{\Isom}{Isom}

\title[RAAGs as normal subgroups of mapping class groups]{Right-angled Artin groups as normal subgroups of Mapping Class Groups}
\author{Matt Clay}
\author{Johanna Mangahas}
\author{Dan Margalit}
\date{} 

\begin{document}

\begin{abstract}
We construct the first examples of normal subgroups of mapping class groups that are isomorphic to non-free right-angled Artin groups.  Our construction also gives normal, non-free right-angled Artin subgroups of other groups, such as braid groups and pure braid groups, as well as many subgroups of the mapping class group, such as the Torelli subgroup.  Our work recovers and generalizes the seminal result of Dahmani--Guirardel--Osin, which gives free, purely pseudo-Anosov normal subgroups of mapping class groups.  We give two applications of our methods: (1)~we produce an explicit proper normal subgroup of the mapping class group that is not contained in any level $m$ congruence subgroup, and (2)~we produce an explicit example of a pseudo-Anosov mapping class with the property that all of its even powers have free normal closure and its odd powers normally generate the entire mapping class group.  The technical theorem at the heart of our work is a new version of the windmill apparatus of Dahmani--Guirardel--Osin, which is tailored to the setting of group actions on the projection complexes of Bestvina--Bromberg--Fujiwara.
\end{abstract}

\maketitle


\vspace*{-1.5em}

\section{Introduction}


This paper is an investigation into the structure of normal subgroups of the mapping class group.  While there is no hope of a complete classification of such subgroups into isomorphism types, we may hope for broader descriptions of the various possible behaviors.   One of the main goals of this paper is to give new examples of right-angled Artin groups that embed as normal subgroups of mapping class groups.    

\subsection{Overview} We begin with an overview of the content---and context---of this paper.  The first examples of normal, right-angled Artin subgroups of mapping class groups of arbitrary surfaces were given in the celebrated work of Dahmani--Guirardel--Osin \cite{ar:DGO17}.  They proved that if $f$ is pseudo-Anosov, then the normal closure of some high power of $f$ is a free group of infinite rank.  

The Dahmani--Guirardel--Osin result cannot be generalized to arbitrary mapping classes of infinite order.  Indeed, if $f$ is a mapping class with sufficiently small support (say, a power of a Dehn twist about a nonseparating curve), then $f$ has a conjugate that commutes with it.  Thus, the normal closure of $f$ is not free.  

In the absence of freeness, we may hope that the normal closure of a power of $f$ is isomorphic to a right-angled Artin group, that is, a group where all of the defining relations are commutations among generators.  However, Brendle and the third author \cite{bm} showed that if the support of a mapping class is sufficiently small (in a precise sense that they define) then its normal closure is not isomorphic---or even abstractly commensurable---to a right-angled Artin group; see also \cite{clayleiningermargalit}. 

In summary: if the support of a mapping class is the entire surface (the pseudo-Anosov case) then a large power has normal closure a free group, and if the support is sufficiently small, then the normal closure cannot be any right-angled Artin group.  The main result of this paper, Theorem~\ref{thm:main} below, can be summarized as:
\begin{quote}
\emph{If the support of a mapping class is sufficiently large, then it has a large power whose normal closure is a right-angled Artin group.}
\end{quote}
As we will see, Theorem~\ref{thm:main} applies not only to single elements but also to finite collections.  Theorem~\ref{thm:main} further applies to normal closures in arbitrary subgroups of the mapping class group as well.  So, for example, we may apply Theorem~\ref{thm:main} in order to construct new normal, right-angled Artin subgroups of the pure braid group and the Torelli group.

Theorem~\ref{thm:main} gives the precise isomorphism types of the right-angled Artin groups that arise from our construction.  Specifically, each is a free product of groups of the following form:
\[
F_\infty, \quad \FreeProdInf (F_\infty \times F_\infty),  \quad \FreeProdInf (F_\infty \times \Z),  \quad \text{and}  \quad \FreeProdInf (F_\infty \times F_\infty \times \Z).
\]
The 14 non-free right-angled Artin groups arising as a free product in this way are the first known examples of non-free, normal, right-angled Artin subgroups of the mapping class group.  We show in Section~\ref{sec:examples} that all 15 right-angled Artin groups described above  (including the free one) also appear as normal subgroups of the Torelli group and the pure braid group.  

As applications of Theorem~\ref{thm:main} we exhibit two new phenomena.
\begin{enumerate}
\item There is a normal subgroup of the mapping class group that is not contained in any level $m$ congruence subgroup.
\item There is a pseudo-Anosov mapping class with the property that all of its odd powers have normal closure equal to the mapping class group and all of its nonzero even powers have normal closure a free group of infinite rank.
\end{enumerate}
See Theorems~\ref{thm:cong} and~\ref{thm:evenodd} below for the precise statements.  The former answers a question raised in earlier work by Lanier and the third author.

In order to prove Theorem~\ref{thm:main} we appeal to---and develop---the theory of projection complexes, introduced by Bestvina--Bromberg--Fujiwara \cite{ar:BBF15}.  We give a general result (Theorem~\ref{thm:spin}), which says that if a group has a ``spinning'' action on a projection complex, then the group is isomorphic to a free product of certain vertex stabilizers.  The projection complexes that we consider in the proof of Theorem~\ref{thm:main} are novel in that the vertices correspond to disconnected subsurfaces of the given surface.  Theorem~\ref{thm:spin} applies much more generally, though, and has applications, for example, to the theory of $\Out(F_n)$ (see Section~\ref{sec:wasf}).  

\ip{Outline of the introduction} The rest of the introduction is structured as follows.  We begin by giving the statement of our main result in Section~\ref{sec:state}.  We give some first consequences in Section~\ref{sec:1st}, where we explicitly construct right-angled Artin subgroups of the mapping class group by directly applying Theorem~\ref{thm:main}.  In Section~\ref{sec:appnc} we discuss Theorem~\ref{thm:cong}, our construction of normal subgroups of the mapping class group that are not contained in congruence subgroups.  In Section~\ref{sec:which} we give a complete picture of which right-angled Artin groups arise from our construction (the 15 examples discussed above).  Then in Section~\ref{sec:whichpowers} we discuss some of the finer details about the statement of our main theorem, addressing the question: given a pseudo-Anosov mapping class, precisely which powers do, and do not, have free normal closure?  Here we state Theorem~\ref{thm:evenodd}.  In Section~\ref{sec:wasf} we discuss our technical results about group actions on projection complexes.  Finally, in Section~\ref{sec:outline} we give an outline of the remainder of the paper.


\subsection{Statement of the main result}\label{sec:state} The statement of our main result, Theorem~\ref{thm:main} below, requires a number of definitions and notations.

\ip{Mapping class groups and pseudo-Anosov mapping classes} Let $S_g$ denote the closed, connected, orientable surface of genus $g$, and let $S_{g,p}$ denote the surface obtained from $S_g$ by deleting $p$ points (so $S_{g,0} = S_g$).  Finally, let $S_{g,p}^b$ be the surface obtained from $S_{g,p}$ by deleting the interiors of $b$ disjoint disks (so $S_{g,p}^0=S_{g,p}$).  The mapping class group $\Mod(S_{g,p}^b)$ is defined as the group of connected components of $\textrm{Homeo}^+(S_{g,p}^b,\partial S_{g,p}^b)$, the group of 
orientation-preserving homeomorphisms of $S_{g,p}^b$ that fix the boundary pointwise.   In what follows, fix $S = S_{g,p}^b$.

The Nielsen--Thurston classification theorem states that each element of $\Mod(S)$ is either periodic, reducible, or pseudo-Anosov; see, e.g. \cite[Chapter 13]{primer}.  The group $\Mod(S)$ acts on the space of projective measured foliations $\PMF(S)$ and an element is pseudo-Anosov if and only if the cyclic subgroup it generates acts with source-sink dynamics; in this case the source is denoted $F_-$ and the sink $F_+$.  

\ip{Partial pseudo-Anosov mapping classes} 
By a subsurface of $S$ we will mean a closed submanifold $X$ with the property that each component of $\partial X$ is either a component of $\partial S$ or an essential, non-peripheral simple closed curve in $S$ (essential means not homotopic to a point or a puncture, and non-peripheral means not homotopic to a component of $\partial S$).  We further assume that no two connected components of $S$ are homotopic to each other (in other words, if $X$ has an annular component then no other component is a parallel annulus).  We also let $\bar X$ denote the surface obtained from $X$ by collapsing each component of the boundary to a marked point (we may alternatively regard the marked points as punctures).

For a subgroup $G$ of $\Mod(S)$, let $\Stab_G(X)$ be the stabilizer in $G$ of the homotopy class of $X$.  There is a well-defined map $\Stab_G(X) \to \Mod(\bar X)$.  We denote the image of $f \in \Stab_G(X)$ by $\bar f$.  

A partial pseudo-Anosov mapping class is an $f \in \Mod(S)$ that has a representative supported on a connected subsurface $X$ and whose image $\bar f$ in $\Mod(\bar X)$ is pseudo-Anosov.  

It follows from the Birman--Lubotzky--McCarthy theory of canonical reduction systems \cite{BLM} that the support $X$ of a partial pseudo-Anosov mapping class $f$ is a well-defined homotopy class of subsurfaces of $S$.

\ip{Independence} We say that pseudo-Anosov mapping classes $f_1,f_2 \in \Mod(S)$ are independent if their corresponding sets of fixed points in $\PMF(\bar S)$ are disjoint (equivalently, not equal).  McCarthy \cite{mccarthy} proved that $f_1$ and $f_2$ are independent in this sense if and only if no nontrivial power of $\bar f_1$ is equal to a power of $\bar f_2$ (if $\partial S = \emptyset$ then $\bar S=S$ and  $\bar f_i = f_i$).   

Suppose that $G$ is a subgroup of $\Mod(S)$.  We further say that $f_1,f_2 \in G$ are $G$--independent if every conjugate of $f_1$ by an element of $G$ is independent of $f_2$ (equivalently, the fixed sets in $\PMF(\bar S)$ lie in different $G$--orbits).  

The definitions of independence and $G$--independence carry over to the case where $f_1$ and $f_2$ are partial pseudo-Anosov elements with the same support $X$.  Specifically, they are independent if the corresponding maps $\bar f_1$ and $\bar f_2$ are independent, and they are $G$--independent if the conjugates by $\Stab_G(X)$ are independent.  

We may further extend the definition to the case where $f_1$ and $f_2$ are arbitrary partial pseudo-Anosov elements.  Specifically, we say that $f_1$ and $f_2$ are $G$--independent if either (1) their supports do not lie in the same $G$--orbit or (2) their supports do lie in the same $G$--orbit and conjugates of $f_1$ and $f_2$ with the same support are $G$--independent as in the previous paragraph.

Even further, let $f_1$ and $f_2$ be two mapping classes and assume that each $f_i$ is either a pseudo-Anosov mapping class, a partial pseudo-Anosov mapping class, or a power of a Dehn twist.  Then we say that $f_1$ and $f_2$ are $G$--independent if either (1) their supports lie in different $G$--orbits or (2) their supports lie in the same $G$--orbit, they are not both powers of Dehn twists, and they are $G$--independent as above.

Finally, if $\F \subseteq \Mod(S)$ is an arbitrary collection of pseudo-Anosov mapping classes, partial pseudo-Anosov elements, and powers of Dehn twists, then we say that $\F$ is $G$--independent if each pair of elements of $\F$ is $G$--independent as above.

\ip{Elementary closure and NEC mapping classes}  Suppose $\partial S = \emptyset$.  Again let $G$ be a subgroup of $\Mod(S)$, and let $f \in G$ be a pseudo-Anosov element.   The elementary closure in $G$ of $f$, written $\EC_G(f)$, is defined to be the stabilizer in $G$ of the associated pair of fixed points in $\PMF(S)$.  For $G = \Mod(S)$ we simply write $\EC(f)$.  We say that $f$ is an NEC element of $G$ if $\grp{f}$ is normal in $\EC_G(f)$. 

McCarthy \cite{mccarthy} proved that $\EC(f)$ is an extension of $\Z$ or $D_\infty$ by a finite subgroup of $\Mod(S)$.  Since the order of a finite subgroup of $\Mod(S)$ is bounded by a function of $S$, it follows that there is a $d=d(S)$ so that the $d$th power of any pseudo-Anosov mapping class is NEC in $\Mod(S)$.  A pseudo-Anosov element that is NEC in $\Mod(S)$ is also NEC in an arbitrary subgroup $G$, and so there is a $d=d(S)$ so that the $d$th power of any pseudo-Anosov element of any subgroup $G$ is an NEC element of $G$.

For the case where $\partial S \neq \emptyset$, we say that a pseudo-Anosov $f \in \Mod(S)$ is NEC if $\bar f \in \Mod(\bar S)$ is.

Finally, we may again carry over the definitions to the partial pseudo-Anosov case.  Let $f \in \Mod(S)$ be a partial pseudo-Anosov element with support $X$.  Consider the map $\Stab_G(X) \to \Mod(\bar X)$.  Denote the image of $\Stab_G(X)$ by $\bar G$ and the image of $f$ by $\bar f$.  We say that $f$ is an NEC element of $G$ if $\bar f$ is an NEC element of $\bar G$.  

A power of a Dehn twist is by definition an NEC element of $\Mod(S)$.

In this paper we give two further sufficient conditions for a pseudo-Anosov mapping class to be an NEC element of $\Mod(S)$.  See the discussion after Proposition~\ref{prop:nec} as well as Lemma~\ref{lem:thurston}.

\ip{Orbit-overlapping subsurfaces...} Let $X$ and $Y$ be two homotopy classes of connected subsurfaces of $S$.  We say that $X$ and $Y$ overlap if the boundary of $X$ has essential intersection with $Y$, and vice versa (meaning that the boundary of every representative of $X$ intersects every representative of $Y$, and vice versa).  

If the boundaries of (representatives of) $X$ and $Y$ have essential intersection then $X$ and $Y$ overlap.  However the converse is not true: consider for example the case where $X$ and $Y$ are the complements of disjoint, homotopically distinct nonseparating annuli in $S$.

We may extend the definition of overlapping to the case of subsurfaces that are not necessarily connected.  We say that two arbitrary subsurfaces $X$ and $Y$ overlap if each each component of $X$ overlaps with each component of $Y$.

Let $G$ be a subgroup of $\Mod(S)$.  We say that a homotopy class of subsurfaces $X$ is $G$--overlapping if for each $h \in G$ we either have that $hX$ is equal to $X$ or it overlaps with $X$.  We say that a family $\X$ of homotopy classes of subsurfaces of $S$ is $G$--overlapping if each $X \in \X$ is $G$--overlapping and for each distinct pair $X,Y \in \X$ any elements of the $G$--orbits of $X$ and $Y$ overlap.  In particular, we must have that $X$ and $Y$ lie in different $G$--orbits. 

One example of a $\Mod(S_g)$--overlapping subsurface of $S_g$ is the complement of a nonseparating annulus.  Another example of a $\Mod(S_g)$--overlapping subsurface of $S_g$ is an annular neighborhood of a separating curve of genus $g/2$, that is, a separating curve that divides $S_g$ into two surfaces of genus $g/2$ (we may also say that the curve itself is $\Mod(S_g)$--overlapping).  Of course if $G_1 \leqslant G_2$ then a $G_2$--overlapping family of subsurfaces is $G_1$--overlapping.  So the given examples are $G$--overlapping for each subgroup $G \leqslant \Mod(S_g)$.  

In general, the $G$--overlapping condition is quite restrictive.  For instance, Corollary~\ref{co:axbx} in Section~\ref{sec:examples} states that if $G=\Mod(S)$ and $X$ is a $G$--overlapping subsurface, then $X$ has at most one annular component and at most two non-annular components.  Proposition~\ref{prop:overlapping} gives further restrictions.

In their work on the large scale geometry of big mapping class groups, Mann--Rafi \cite{un:MR} define the notion of a nondisplaceable subsurface of a surface, which is an analogue of our $G$--overlapping condition in the case that $G=\Mod(S)$.  For a connected subsurface, their notion is the same as our notion of $\Mod(S)$--overlapping, but for disconnected subsurfaces it is less restrictive.

\ip{...and the associated mapping classes}  As above, let $G$ be a subgroup of $\Mod(S)$.  Let $\X$ be a $G$--overlapping family of homotopy classes of subsurfaces of $S$.  We say that a family of mapping classes $\F \subseteq G$ is carried by $\X$ if there is a function
\[
\sigma \from \F \to \X
\]
so that the following conditions hold: 
\begin{enumerate}
\item each $f \in \F$ is a pseudo-Anosov mapping class, partial pseudo-Anosov mapping class, or a power of a Dehn twist, and the support of $f$ is (the homotopy class of) a component of (a representative of) $\sigma(f)$, and
\item for any component $X$ of an element of $\X$, there is an $f \in \F$ and $g \in G$ so that the support of $gfg^{-1}$ is $X$.
\end{enumerate}
In the case where no two components of elements of $\X$ lie in the same $G$--orbit, we may replace condition (2) with the simpler condition that for any component $X$ of an element of $\X$, there is an $f \in \F$ whose support is $X$.  Condition (2) is required for the situation where an element of $\X$ has two components $X_1$ and $X_2$ in the same $G$--orbit; say $gX_1=X_2$.  In this case, if the family $\F$ had both an element $f_1$ with support $X_1$ and the element $f_2 = gf_1g^{-1}$ with support $X_2$ (as per the simpler condition), then $\F$ would not be $G$--independent.

\ip{Statement of the main theorem} In order to state our main theorem, we need several more notations.  First, for a subset $A$ of a group $G$ we denote by $A^{(n)}$ the set $\{a^n \mid a \in A\}$.  Also, if $G$ is a group and $A \subseteq G$ we denote by $\normalclosure{A}_G$ the normal closure of $A$ in $G$.  

Suppose that $G$ is a subgroup of $\Mod(S)$, that $\X$ is a $G$--overlapping family of homotopy classes of subsurfaces of $S$, and that $\F \subseteq G$ is a family of mapping classes carried by $\X$.  Also suppose that $Y = hX$ for some $X \in \X$, $h \in G$.  For $n > 0$ we define
\[
R_Y = h \left(\normalclosure{\left(\sigma^{-1}(X)\right)^{(n)}}_{\Stab_G(X)}\right) h^{-1}.
\]
This subgroup of $G$ is well-defined, independent of $h$.  For $X \in \X$, the group $R_X$ is the normal closure in $\Stab_G(X)$ of the $n$th powers of the elements of $\F$ that are supported in $X$.  For $Y = hX$ we have $R_Y$ is the conjugate of $R_X$ by $h$.  We emphasize that the notation $R_Y$ does not reflect the dependence on $\F$ and $n$.

Finally, we denote by $a_Y$ and $b_Y$ the numbers of annular and non-annular components of a subsurface $Y \subseteq S$ (or a homotopy class).

\begin{theorem}
\label{thm:main}
Let $S = S_{g,p}^b$, let $G$ be a subgroup of $\Mod(S)$, let $\X$ be a $G$--overlapping family of subsurfaces of $S$, and let $\F$ be a finite, $G$--independent family of mapping classes that are carried by $\X$ and are NEC in $G$.  There is an $N > 0$ with the following properties:
\begin{enumerate}
\item for each $n \geq N$ and any set of orbit representatives $\calY$ for the action of $\normalclosure{\F^{(n)}}_{G}$ on $G \cdot \X$ we have
\begin{equation*}
\normalclosure{\F^{(n)}}_{G} \cong \BigFreeProd{Y \in \calY} R_Y, \  and
\end{equation*} 
\item further, for each $Y \in \calY$ we have
\begin{equation*}
R_Y \cong F_{\infty}^{b_Y} \times \Z^{a_Y}.	
\end{equation*}
\end{enumerate}
\end{theorem}

Since every $f \in \Mod(S)$ has a power $p$ that is NEC in $G$, and since a power of an NEC mapping class is NEC, we may remove the NEC hypothesis from Theorem~\ref{thm:main} at the expense of replacing ``for all $n > N$'' with ``there exists an $N$ so that for all multiples $n$ of $N$''; specifically, $N$ is chosen so that each element of $\F^{(N)}$ is NEC in $G$.  In other words, if we remove the NEC hypothesis, the conclusion holds for some specific $N$ (and its multiples) instead of all sufficiently large $n$.

It follows from the proof that there exists a set of orbit representatives $\calY$ so that the isomorphism in the first statement of Theorem~\ref{thm:main} is given by inclusion.

\ip{Connection to the deep relations question of Ivanov} Let $\T_k(S)$ denote the normal subgroup of $\Mod(S)$ generated by all $k$th powers of Dehn twists.  Ivanov asked \cite[\S 12]{ivanov15} if $\T_k(S)$ has a presentation where the generators are all $k$th powers of Dehn twists and the relations are the obvious ones, namely the relations
\[
T_d^kT_c^kT_d^{-k} = T_{T_d^k(c)}^k.
\]
This question was recently answered in the affirmative by Dahmani~\cite{ar:Dahmani18}.  In the proof, Dahmani applies a version of the Dahmani--Guirardel--Osin machinery.

One interpretation of our Theorem~\ref{thm:main} is that there is a presentation for the group $\normalclosure{\F^{(n)}}_{G}$ where the set of generators is the union of all $G$--conjugates of $\F^{(n)}$, and where the relations are the obvious ones.  Specifically, the generators are of the form $hfh^{-1}$ with $f \in \F^{(n)}$ and $h \in G$.  And, denoting each such generator $hfh^{-1}$ as $f_h$, the relations are all equalities of the form
\[
g_k f_h g_k^{-1} = f_{g_kh},
\]
These relations are indeed obvious, as can be seen by expanding them out.  Once we have the presentation of the normal closure given by Theorem~\ref{thm:main}, we can obtain the presentation here by repeatedly performing the Tietze transformation of adding a new generator and writing it in terms of the old generators.  In this sense, our result is analogous to Dahmani's.  

\begin{figure}[ht]
\centering
\begin{tikzpicture}[scale=.65]
\def\s{white!85!blue}
\filldraw[\s] (-4,-1.5) rectangle (-0.85,1.5);
\filldraw[\s] (4,-1.5) rectangle (0.86,1.5);
\filldraw[\s] (-0.34,-1.5) rectangle (0.35,1.5);
\draw[thick,fill=\s] (4,1.5) .. controls (8,1.25) and (8,-1.25) .. (4,-1.5);
\draw[thick,fill=\s] (-4,1.5) .. controls (-8,1.25) and (-8,-1.25) .. (-4,-1.5);
\fill[\s] (-4,1.5) .. controls (0,1.65) .. (4,1.5);
\fill[\s] (-4,-1.5) .. controls (0,-1.65) .. (4,-1.5);
\fill[white] (-0.84,1.6) rectangle (-0.35,-1.6);
\fill[white] (0.84,1.6) rectangle (0.35,-1.6);
\draw[thick] (-4,1.5) .. controls (0,1.65) .. (4,1.5);
\draw[thick] (-4,-1.5) .. controls (0,-1.65) .. (4,-1.5);
\foreach \a in {-3.75,3.75} {
	\begin{scope}[xshift=\a cm]
	\draw[thick] (-2.2,-0.25) .. controls (-1.5,-1) and (-0.3,0) .. (0,-0.8);
	\draw[thick] (2.2,-0.25) .. controls (1.5,-1) and (0.3,0) .. (0,-0.8);
	\end{scope}
	\node (g) at (\a,-1.15) {$g/2$};
}
\foreach \a in {-5.5,-2,2,5.5}{
	\begin{scope}[xshift=\a cm,yshift=0.15cm,rotate=90]
	\begin{scope}
	\clip (-0.13,1) rectangle (-1,-1);
	\draw[thick,fill=white] (0,0.5) arc (90:270:0.25cm and 0.5cm);	
	\end{scope}
	\draw[thick,fill=white] (-0.13,-0.43) arc (-90:90:0.13cm and 0.43cm);	
	\draw[thick] (0,0.5) arc (90:270:0.25cm and 0.5cm);
	\end{scope}
}
\foreach \a in {3.25,3.75,4.25}{	
	\filldraw (\a,0.05cm) circle (0.025);
	\filldraw (-1*\a,0.05cm) circle (0.025);
}
\draw[thick,fill=\s] (-0.85,1.62) to[out=275,in=85] (-0.85,-1.62);
\draw[thick,fill=white] (-0.35,1.62) to[out=275,in=85] (-0.35,-1.62);
\draw[thick,fill=\s] (0.35,1.62) to[out=275,in=85] (0.35,-1.62);
\draw[thick,fill=white] (0.85,1.62) to[out=275,in=85] (0.85,-1.62);
\end{tikzpicture}
\caption{An orbit overlapping subsurface in $S_g$}\label{fig:g/2}
\end{figure}
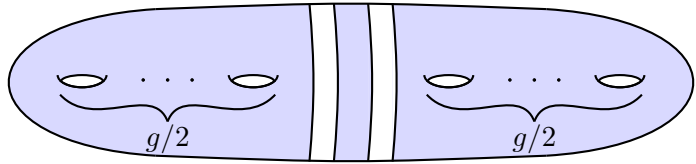

\subsection{First examples}\label{sec:1st} We list here some immediate consequences of Theorem~\ref{thm:main}, which exhibit some of the variety of applications.  In what follows, $\I(S_g)$ is the Torelli subgroup of $\Mod(S_g)$, which is defined to be the kernel of the action of $\Mod(S_g)$ on $H_1(S_g;\Z)$.  Also, we identify the pure braid group on $n$ strands with the pure mapping class group of a disk $D_n$ with $n$ marked points in the interior; see \cite[Chapter 9]{primer}.
\begin{enumerate}
\item Taking $\X = \{S\}$ and $G = \Mod(S)$ we obtain the result of Dahmani--Guirardel--Osin that the normal closure of a suitable power of a pseudo-Anosov mapping class is isomorphic to $F_\infty$.
\item Taking $\X = \{X\}$ where $X$ is connected and $\Mod(S)$--overlapping and taking $G=\Mod(S)$, we obtain that the normal closure of a suitable power of a partial pseudo-Anosov mapping class with support $X$ is isomorphic to $F_\infty$.
\item Taking $\X = \{A\}$ where $A$ is an orbit-overlapping annulus in $S$---for example an annulus dividing $S$ into two homeomorphic subsurfaces as in Figure~\ref{fig:g/2}---and taking $G = \Mod(S)$, we obtain that the normal closure of a suitable power of a Dehn twist about $A$ is isomorphic to $F_\infty$.
\item Taking $\X = \{A \cup B\}$ where $A$ is as in example (3) and $B$ is the complement of an open neighborhood of $A$ (see Figure~\ref{fig:g/2}), taking $f_A$ to be a suitable power of a Dehn twist about $A$, taking $f_B$ to be a suitable power of a partial pseudo-Anosov mapping class supported on one component of $B$, and taking $G = \Mod(S)$, we obtain that the normal closure of $\{f_A,f_B\}$ in $\Mod(S)$ is isomorphic to 
\[
\FreeProdInf \left(F_{\infty} \times F_{\infty} \times \Z\right).
\]
\item Taking $\X = \{A\}$ where $A$ is any separating annulus in $S_g$  and taking $G = \I(S_g)$, we obtain that the normal closure in $\I(S_g)$ of a suitable power of a Dehn twist about $A$ is isomorphic to $F_\infty$.  Similarly the normal closure in $\I(S_g)$ of a suitable power of any partial pseudo-Anosov element of $\I(S_g)$ is isomorphic to $F_\infty$. 
\item Taking $\X=\{A\}$ to be any subsurface of $D_n$ and taking $G$ to be the pure braid group, and taking $f$ to be any partial pseudo-Anosov element or Dehn twist supported on $A$, we obtain that the normal closure of a suitable power of $f$ is isomorphic to $F_\infty$.  
\item Taking $\X=\{A \cup B \cup C\}$ where $A$ is an annulus in $D_n$ surrounding more than two marked points but fewer than $n-1$ marked points, taking $B$ and $C$ to be the complementary regions to open neighborhood of $A$, and taking $G$ to be the pure braid group, and taking $f_A$, $f_B$, and $f_C$ to be partial pseudo-Anosov elements and Dehn twists with supports equal to $A$, $B$, and $C$, we obtain that the normal closure of suitable powers of $f_A$, $f_B$, and $f_C$ is isomorphic to
\[\FreeProdInf \left(F_\infty \times F_\infty \times \Z\right). \] 
\item Taking $\X = \{A_1 \cup \cdots \cup A_n\}$ where the $A_i$ are pairwise disjoint and each $A_i$ is the support of a partial pseudo-Anosov mapping class $f_i$ and taking $G$ to be the subgroup of $\Mod(S)$ consisting of elements that preserve each $A_i$, we obtain that the normal closure in $G$ of suitable powers of the $f_i$ is isomorphic to the direct product of $n$ copies of $F_\infty$.
\end{enumerate}

\subsection{Application: non-congruence normal subgroups}\label{sec:appnc} We give here an application of Theorem~\ref{thm:main} to the general theory of normal subgroups of the mapping class group.  The following theorem answers a question asked in (the first version of) a paper by Lanier and the third author \cite{lmv1}.

\begin{theorem}
\label{thm:cong}
Let $g \geq 2$.   There exists a proper normal subgroup of $\Mod(S_g)$ that is not contained in any proper level $m$ congruence subgroup of $\Mod(S_g)$.
\end{theorem}

To prove Theorem~\ref{thm:cong}, we first choose independent NEC pseudo-Anosov mapping classes $f_1$ and $f_2$, each of whose actions on $H_1(S_g;\Z)$ is equal to that of a Dehn twist about some nonseparating curve.  Then for distinct  large primes $p_1$ and $p_2$, the normal closure $N$ of $\{f_1^{p_1},f_2^{p_2}\}$ is the desired subgroup.  Indeed, Theorem~\ref{thm:main} implies $N$ is free (hence proper) and it is evident that $N$ is not contained in any level $m$ congruence subgroup.  The main work is to prove that $f_1$ and $f_2$ exist.  In the process, we make $f_1$ and $f_2$ (but not $p_1$ and $p_2$) explicit.

After learning about our work, Ashot Minasyan pointed out to us that Theorem~\ref{thm:cong} can also be derived from earlier work of Hull \cite{ar:Hull16}.  We explain the details of this argument in Section~\ref{sec:applications}. 

\subsection{Which RAAGs?}\label{sec:which} Putting the two statements of Theorem~\ref{thm:main} together and applying the restrictions on $a_Y$ and $b_Y$ from Corollary~\ref{co:axbx} (mentioned above), we see that when $G = \Mod(S)$ the group $\normalclosure{\F^{(n)}}_G$ is isomorphic to a free product of the groups \begin{equation*}
F_{\infty}, \quad \FreeProdInf (F_{\infty} \times \Z), \quad \FreeProdInf (F_{\infty} \times F_{\infty}), \ \ \text{and} \ \ \FreeProdInf (F_{\infty} \times F_{\infty} \times \Z);
\end{equation*}
We prove in Section~\ref{sec:examples} that all possibilities occur, that is, any group that is a free product of groups \[
\FreeProdInf F_\infty^{b_i} \times \Z^{a_i},
\]
where $a_i \in \{0,1\}$ and $b_i \in \{0,1,2\}$ for all $i$, is isomorphic to some $\normalclosure{\F^{(n)}}_{\Mod(S)}$; see Theorem~\ref{thm:examples}.  In particular, this gives the first examples of normal, non-free, right-angled Artin subgroups of the mapping class group.

Theorem~\ref{thm:main} applies to mapping classes whose supports are ``sufficiently large.''  On the other hand, a result of Brendle and the third author roughly states that the normal closure of any mapping class with ``sufficiently small'' support is not isomorphic to a right-angled Artin group \cite[Corollary 1.4]{bm}.  This leads us to the following conjecture.

\begin{conjecture}
Let $S=S_{g,p}$.  If $N$ is a nontrivial normal subgroup of $\Mod(S)$ and $N$ is isomorphic to a right-angled Artin group then $N$ is isomorphic to one of the right-angled Artin subgroups of $\Mod(S)$ afforded by our construction.  In particular, $N$ is isomorphic to a free product of the groups
\[ F_\infty, \qquad \FreeProdInf \left(F_\infty \times \Z\right), \qquad \FreeProdInf \left(F_\infty \times F_\infty\right),\quad  \text{ and } \quad \FreeProdInf \left(F_\infty \times F_\infty \times \Z\right).\]
\end{conjecture}

For specific values of $g$ and $n$ the conjecture may be sharpened.  For instance, when $S=S_g$ the conjecture says that the right-angled Artin group is isomorphic to a free product of the groups
\[ F_\infty, \qquad \FreeProdInf \left(F_\infty \times F_\infty\right),\quad  \text{ and } \quad \FreeProdInf \left(F_\infty \times F_\infty \times \Z\right),\]
since the group $\FreeProdNoArg (F_\infty \times \Z)$ does not arise from our construction in this case (this is a consequence of Proposition~\ref{prop:overlapping}); cf. Theorem~\ref{thm:examples2}.  Similarly, when at least one of $g$ or $n$ is odd, the group $\FreeProdNoArg (F_\infty \times F_\infty \times \Z)$ does not appear, etc.


\ip{Other normal free groups...} Farb explicitly asked whether or not a pseudo-Anosov mapping class has a power with free normal closure \cite[Question 2.9]{farbprobs}.  Before the work of Dahmani--Guirardel--Osin, the only known examples of free, normal subgroups of the mapping class group were the ones due to Whittlesey \cite{whittle}.  She proved that the Brunnian subgroup of $\Mod(S_{0,n})$ is free.  As a consequence, she proved that there is a corresponding free, normal, all pseudo-Anosov subgroup of $\Mod(S_{2})$.

\ip{...and non-normal RAAGs}  Long before the work of Dahmani--Guirardel--Osin, it was proven that if two curves have geometric intersection number greater than 1 then the corresponding Dehn twists generate a free group of rank 2.  This result was proved by Ishida \cite{ishida} and Hamidi--Tehrani \cite{hht}, and also appears in Handel's notes from Thurston's course on mapping class groups in Princeton from 1975.  Ivanov \cite[Corollary 8.4]{ivanovbook} and McCarthy \cite{mccarthy} proved that high powers of independent pseudo-Anosov mapping classes generate a free group of rank 2.  Since then there have been many different constructions of non-normal free subgroups of the mapping class group with various properties; see, for instance, the work of Fujiwara \cite{kf} and of the second author \cite{ar:Mangahas10,jm1}.

Beyond free groups, there are a number of papers devoted to the problem of finding right-angled Artin subgroups of mapping class groups.  Constructions of such subgroups are given by Leininger and the first two authors of this paper \cite{CLM}, Crisp--Farb \cite{crispfarb}, Crisp--Paris \cite{crispparis}, Crisp--Wiest \cite{crispwiest}, Koberda \cite{koberda}, L\"onne \cite{lonne}, Runnels \cite{Runnels}, and Seo \cite{Seo}.  Each of these works has its own points of emphasis, but one over-arching theme is that every finitely generated right-angled Artin group is isomorphic to a subgroup of some mapping class group.  In all of these works, the resulting subgroup is not normal.

\subsection{Which powers?}\label{sec:whichpowers} As mentioned, the special case of Theorem~\ref{thm:main} where each element of $\F$ is pseudo-Anosov is due to Dahmani--Guirardel--Osin.  The most obvious distinctions between their work and ours are that Theorem~\ref{thm:main} applies to certain types of reducible mapping classes, and also that our normal closures are not always free groups.  Beyond this, we help clarify the situation as to which powers of a pseudo-Anosov mapping class do, and do not, have free normal closure.  We now discuss these two points in more detail.

\ip{Which powers do have free normal closure?} In a forthcoming paper \cite{powers} we apply the techniques of this paper to show the following.
\begin{enumerate}
\item If $f \in \Mod(S_{g,p})$ is an NEC pseudo-Anosov element then $\normalclosure{f^N}$ is free for \[N \geq \mathrm{exp}\left(\mathrm{exp}\left(10^6\cdot \delta_{g,p}^2\cdot(3g-3+p)^2\right)\right).\]
\item If $c$ is an orbit overlapping curve and $p=0$ then $\normalclosure{T_c^N}$ is free for $N \geq 53,489$.
\end{enumerate}
Here, $\delta_{g,p}$ is any hyperbolicity constant for the curve graph of $S_{g,p}$; it has been shown that $\delta_{g,p}$ can be taken to be 17, independently of $S_{g,p}$~\cite{ar:HPW15}.

It is an interesting problem to sharpen these values of $N$.  For the case of an orbit overlapping curve, we can see that $N=1$ does not suffice to ensure that $\normalclosure{T_c^N}$ is free.  For example, when $c$ is a curve of genus $g/2$ in $S_g$, the lantern relation can be used to exhibit the non-freeness (if $a$, $b$, and $c$ are curves of genus $g/2$ that lie in a 4-holed sphere $L$ in $S_g$, then by the lantern relation we have $T_aT_b=T_c^{-1}M$ where $M$ is the multitwist about the boundary of $L$, from which we see that $T_aT_b$ and $T_c$ commute but do not have a common power).  On the other hand, we do not know whether or not $\normalclosure{T_c^2}$ is free.  Similarly, $N=1$ does not suffice in the pseudo-Anosov case: as explained below in Proposition~\ref{prop:nec}, every pseudo-Anosov element of the Torelli group $\I(S_g)$ is NEC, and Lanier and the third author proved \cite[Theorem~1.3]{lm} that there is a  pseudo-Anosov element of $\I(S_g)$ whose normal closure in $\Mod(S_g)$ is equal to $\I(S_g)$.

\ip{Which powers do not have free normal closure?}  It follows from the work of Dahmani--Guirardel--Osin---specifically Corollary 6.36 and Theorem 6.8 in their paper \cite{ar:DGO17}---that sufficiently large powers of NEC pseudo-Anosov mapping classes have free normal closure (this is also a special case of Theorem~\ref{thm:main}).  Many results about pseudo-Anosov mapping classes hold for sufficiently large powers---for example the result of Ivanov and McCarthy mentioned above.  Thus, one may be tempted to think that  all sufficiently large powers of a pseudo-Anosov mapping class have free normal closure.  The following proposition shows that this is not the case.

\begin{proposition}
\label{prop:nec}
Let $f \in \Mod(S)$ be a pseudo-Anosov mapping class that is not NEC.  Then for arbitrarily large $n$, the group $\normalclosure{f^n}$ contains a nontrivial periodic element; in particular, $\normalclosure{f^n}$ is not free.
\end{proposition}

Indeed, any $h \in \EC(f)$ conjugates $f$ to $hfh^{-1}=f^{\pm 1}r$ for some $r$ in the finite group of symmetries of the singular Euclidean structure on $S$ associated to $f$, as these are precisely the pseudo-Anosov elements in $\EC(f)$ with the same stretch factor as $f$ (this follows from McCarthy's description of $\EC(f)$ in his unpublished paper \cite{mccarthy}).  If $f$ is not NEC, then some $h$ gives a nontrivial $r = f^{\mp 1}hfh^{-1}$ in $\normalclosure{f}$.  Moreover, if $f$ is not NEC then arbitrarily large powers of $f$ are not NEC, hence the proposition. 

It follows immediately from Proposition~\ref{prop:nec} that if $f$ lies in a normal subgroup of $\Mod(S)$ that is torsion free, then $f$ is NEC.  As an example, if $f$ lies in the level $m$ congruence subgroup of $\Mod(S_g)$ with $m \geq 3$, then $f$ is NEC in $\Mod(S_g)$.  In particular, every pseudo-Anosov element of the Torelli group $\I(S_g)$ is NEC.  

Proposition~\ref{prop:nec} shows that the NEC condition is necessary in order for a pseudo-Anosov mapping class to have free normal closure.  However it is not sufficient.  Indeed, as mentioned above Lanier and the third author gave an example of a NEC pseudo-Anosov mapping class with normal closure $\I(S_g)$.  On the other hand, the following question seems to be open: is the normal closure of an NEC pseudo-Anosov mapping class torsion free?

In the case of a closed surface, we can strengthen the conclusion of Proposition~\ref{prop:nec}: not only does $\normalclosure{f^n}$ fail to be free, but it fails to be abstractly commensurable to any right-angled Artin group.  Indeed, it follows from work of Lanier and the third author \cite[Theorem 1.1]{lm} that the normal closure of any nontrivial periodic element of $\Mod(S_g)$---hence the normal closure of $f^n$---contains $\I(S_g)$.  Also, Brendle and the third author proved that the normal closures of any subgroup of $\Mod(S_g)$ containing $\I(S_g)$ is not abstractly commensurable with any right-angled Artin group \cite[Corollary 1.4]{bm}. 

The failure of $\normalclosure{f^n}$ to be free (or abstractly commensurable to any right-angled Artin group) is underscored by the following result, which we prove in Section~\ref{sec:applications}.

\begin{theorem}
\label{thm:evenodd}
For each $g \geq 3$ there exists a pseudo-Anosov $f \in \Mod(S_g)$ so that if $n$ is odd then
\begin{align*}
\normalclosure{f^n} &= \Mod(S_g), \\
\intertext{and if $n$ is even and nonzero then}
\normalclosure{f^n} &\cong F_\infty. 
\end{align*} 
\end{theorem}

The mapping classes that appear in the proof of Theorem~\ref{thm:evenodd} are based on the ones used by Lanier and the third author to show that there are pseudo-Anosov elements of $\Mod(S_g)$ with the property that all of their odd powers are normal generators (in particular, there are normal generators for $\Mod(S_g)$ with arbitrarily large translation lengths on the curve complex) \cite[Theorem 1.4]{lm}.


\subsection{Windmills and spinning families}\label{sec:wasf} Our next goal is to explain the main technical theorem used to prove Theorem~\ref{thm:main}, namely, Theorem~\ref{thm:spin} below.  This theorem concerns the theory of group actions on projection complexes.  Briefly, a projection complex is a graph $\Gamma$ that comes equipped with a collection of  functions
\[
d_v \from V \setminus \{v\} \times V \setminus \{v\} \to \R_{\geq 0}
\]
where $V$ is the set of vertices of $\Gamma$ and $v \in V$.   Projection complexes were defined by Bestvina--Bromberg--Fujiwara; see Section~\ref{sec:back} for the full definition, along with examples and motivation.  Our definition is a mild modification of the original definition of Bestvina--Bromberg--Fujiwara in that we require the projection complex to satisfy additional properties (such as the bounded geodesic image property).

After stating Theorem~\ref{thm:spin}, we explain forthcoming work on the geometry of the corresponding quotient complexes, and then we state Theorem~\ref{thm:base}.  The latter is the simplest type of application of Theorem~\ref{thm:spin}, where the output is a free group.  In Section~\ref{sec:outline} we explain how Theorems~\ref{thm:spin} and~\ref{thm:base} are pieced together to prove Theorem~\ref{thm:main}.

\ip{Windmills in projection complexes}  Let $\P$ be a projection complex, and let $G$ be a group that acts on $\P$.  Further, for each vertex $v$ of $\P$, let $R_v$ be a subgroup of the stabilizer of $v$ in $G$.  Let $L > 0$.  We say that the family of subgroups $\{R_v\}$ is an \emph{equivariant $L$--spinning family} of subgroups of $G$ if it satisfies the following two conditions:

\begin{itemize}
\item \emph{Equivariance:} If $g$ lies in $G$ and $v$ is a vertex of $\P$ then
\[
gR_vg^{-1} = R_{gv}.
\]  
\item \emph{Spinning condition:} For any distinct vertices $v$ and $w$ of $\P$ and any nontrivial $h \in R_v$ we have
\[
d_v(w,hw) \geq L.
\]
\end{itemize}
The equivariance condition implies that for each vertex $v$ the subgroup $R_v$ is normal in $\Stab(v)$, and that the subgroup $\langle R_v \rangle$ of $G$ generated by the $R_v$ is normal in $G$.  Furthermore, if we choose orbit representatives $\{v_i\}$ for the action of $G$ on the vertices of $\P$, then $\langle R_v \rangle$ is the normal closure of the set $\left\{R_{v_i}\right\}$. 

\begin{theorem}\label{thm:spin}
Let $\P$ be a projection complex and let $G$ be a group acting on $\P$.  There exists a constant $L(\P)$ with the following property.  If $L \geq L(\P)$,  if $\{R_v\}$ is an equivariant $L$--spinning family of subgroups of $G$, and if $\calO$ is any set of orbit representatives for the action of $\langle R_v\rangle$ on the set of vertices of $\P$, then $\grp{R_v}$ is isomorphic to the free product
\[
\BigFreeProd{v\in \calO} R_v.
\]
\end{theorem}

In the statement of Theorem~\ref{thm:spin}, we emphasize that $\langle R_v\rangle$ is the group generated by all of the $R_v$, not just one $R_v$.  Also, it follows from the proof that there exists a set of orbit representatives $\calO$ so that the isomorphism in Theorem~\ref{thm:spin} is given by inclusion.

To prove Theorem~\ref{thm:spin} we introduce the notion of a windmill in a projection complex; see Section~\ref{sec:windmillsprojection}.  A windmill is a subgraph of $\P$ that serves as a proxy for the Bass--Serre tree for the desired free product decomposition.  Our definition of a windmill is derived from the theory of windmills for group actions on hyperbolic spaces, due to Dahmani--Guirardel--Osin, which in turn has its origins in the work of Gromov \cite[\S 26]{gromov}.  

As we explain in Section~\ref{sec:back}, every simplicial tree can be regarded as a projection complex (and conversely a projection complex is a quasi-tree~\cite[Theorem~3.16]{ar:BBF15}).  In Section~\ref{sec:trees} we give the proof of Theorem~\ref{thm:spin} in the special case where $\P$ is a tree, as a warmup for the general case.

After the first version of this paper appeared, Bestvina--Dickmann--Domat--Kwak--Patel--Stark gave a new proof of Theorem~\ref{thm:spin}.  Their proof produces an $\langle R_v \rangle$--invariant tree in $\P$ that is the Bass--Serre tree for the given free product decomposition.  In their paper \cite{un:BDDKPS}, they also explain how to derive a version of the main result of the paper by Dahmani--Guirardel--Osin \cite[Theorem 5.3a]{ar:DGO17} from Theorem~\ref{thm:spin}.

\ip{Hyperbolicity of the quotient} As above $\T_k(S)$ denotes the normal subgroup of $\Mod(S)$ generated by $k$th powers of Dehn twists.  Building on the aforementioned work of Dahmani, it was recently shown by Dahmani--Hagen--Sisto that for suitable $k$ the quotient group $\Mod(S)/\T_k(S)$ is acylindrically hyperbolic \cite{un:DHS} (in particular the group $\T_k(S)$ has infinite index in $\Mod(S)$).  One of the major steps in the proof is to show that the quotient of the curve complex by $\T_k(S)$ is hyperbolic.  

In analogy with this result, the first two authors prove in a separate paper that under certain hypotheses, the quotient complex $\P/\grp{R_{v}}$ arising in Theorem~\ref{thm:spin} is hyperbolic and that the quotient group $G/\grp{R_{v}}$ is acylindrically hyperbolic \cite{quotient}.  The argument follows along the lines of the work of  Dahmani--Hagen--Sisto~\cite{un:DHS}.

\ip{Free groups from windmills} Our next result, Theorem~\ref{thm:base} below, is the simplest type of application of Theorem~\ref{thm:spin}, in that the $R_v$ subgroups are all isomorphic to $\Z$ (and so the output is $F_\infty$).  Theorem~\ref{thm:base} also serves as a sort of base case for Theorem~\ref{thm:main}; in the outline of the paper given below, we explain this in more depth.

For the statement we require several definitions.  First, when we say that a space is hyperbolic, we mean that it is a metric space that is $\delta$--hyperbolic in the sense of Gromov.  Fix some hyperbolic space $X$ and suppose a group $G$ acts on $X$ by isometries.  Certain types of elements of $G$ are called WPD elements (for ``weakly properly discontinuous'').  The idea is that $f$ is WPD if its translation length is positive and if long segments of a quasi-axis for $f$ have a finite coarse stabilizer; see Section~\ref{sec:eg} for the details.  One important example for our work is where $G = \Mod(S)$ and $X = \C(S)$; in this case the WPD elements are exactly the pseudo-Anosov elements of $\Mod(S)$.  

Each WPD element $f \in G$ has two fixed points in $\partial X$, and so the elementary closure can be defined for $f$ in the same way as for pseudo-Anosov elements of $\Mod(S)$: it is the stabilizer of this pair of points in $\partial X$.  We say that $f$ is NEC if it is normal in its elementary closure.  Also, we say that $f_1$ and $f_2$ are independent if $\EC(f_{1}) \cap \EC(f_{2})$ is finite and they are normally independent if every conjugate of $f_1$ is independent of $f_2$.  Finally, a family of WPD elements is normally independent if they are pairwise normally independent.

\begin{theorem}\label{thm:base}
Suppose $G$ acts on a hyperbolic space $X$ and 
$\{f_1,\ldots,f_m\} \subseteq G$
is a normally independent collection of NEC WPD elements of $G$.  For any $t \geq 0$ there is a constant $N$ such that for $n \geq N$ the group
\begin{equation*}
\normalclosure{f_1^n,\dots,f_m^n}_{G}
\end{equation*}
has the following properties:
\begin{enumerate}
\item it is isomorphic to $F_\infty$ with a free basis consisting of conjugates of the $f_i^n$, and
\item the translation length of each nontrivial element is at least $t$.
\end{enumerate}
\end{theorem}

As before, every WPD element has a power that is NEC and so Theorem~\ref{thm:base} may be applied to any normally independent collection of WPD elements after replacing each by a power.

Besides the mapping class group, another application of Theorem~\ref{thm:base} is to the outer automorphism group of a finitely generated free group $F_n$.  In this case the space $X$ is the free factor complex of $F_n$, and the collection $\{f_{1},\ldots,f_{m}\} \subseteq \Out(F_n)$ is any collection of normally independent fully irreducible outer automorphisms.  See the paper by Bestvina--Feighn~\cite{ar:BF14-1} for the relevant definitions.


\subsection{Outline of the paper}
\label{sec:outline}

We begin in Section~\ref{sec:examples} by describing the basic properties of orbit-overlapping families of subsurfaces.  In particular, in Section~\ref{sec:overlap} we classify all orbit-overlapping families for an arbitrary normal subgroup of the mapping class group.  We also show in Theorem~\ref{thm:examples} that all of the right-angled Artin groups discussed above do indeed appear as normal subgroups of some mapping class group.  Putting these two results together, we obtain a complete classification of right-angled Artin subgroups than can arise from our construction.  We also show that the same right-angled Artin groups appear as normal subgroups of pure braids groups and Torelli groups.

In Section~\ref{sec:back} we give the definition of a projection complex.  After explaining the basic examples, we present a new example based on disconnected subsurfaces that is used for the proof of Theorem~\ref{thm:main}.  We conclude this section with a construction of Bestvina--Bromberg--Fujiwara and Dahmani--Guirardel--Osin that allows us to build a projection complex from a group action on a hyperbolic space.  For the proof of Theorem~\ref{thm:base} we apply their construction to the case of the $\Mod(S)$ action on the curve graph $\C(S)$; the key fact about this action that is needed for the construction is that the action of a pseudo-Anosov mapping class on $\C(S)$ is WPD.  

Section~\ref{sec:windmillsprojection} provides the architecture for the proof of Theorem~\ref{thm:spin}, and gives two successive reductions of Theorem~\ref{thm:spin}.  We begin by defining the windmill associated to the action of a group $G$ on a projection complex $\P$.  A windmill is the direct limit of an increasing union of certain subgraphs $W_i$.  We then define an increasing sequence of free products $F_i$, each equipped with a homomorphism to $G$ and an action on $W_i$.  Proposition~\ref{prop:rephrase} states that Theorem~\ref{thm:spin} holds if each homomorphism $F_i \to G$ is injective; this is the first reduction.  A natural way to prove the injectivity is to show that the induced action on $\P$ is faithful.  In Section~\ref{sec:windmillsprojection} we introduce the notions of pivot points and waypoints as tools for showing that this action is faithful (Proposition~\ref{prop:pivotway}).

In Section~\ref{sec:trees} we use the machinery developed in Section~\ref{sec:windmillsprojection} to prove Theorem~\ref{thm:spin} in the special case where the projection complex is a tree.  The special case is stated as Theorem~\ref{thm:treecase}.  While this theorem is an immediate consequence of the Bass--Serre theory for group actions on trees, the proof serves as a demonstration of the windmill machinery in action, and also serves as a template for the proof of Theorem~\ref{thm:spin} in the following section.  

Section~\ref{sec:proof} contains the proof of Theorem~\ref{thm:spin}, our main technical result about projection complexes.  The proof is analogous to the proof of Theorem~\ref{thm:treecase}, but it is more complicated.  The main new difficulty is that the graphs $W_i$ used to define the windmill are not necessarily convex (in a tree any connected subset is convex).

We next turn toward the proof of Theorem~\ref{thm:main}.  First, in Section~\ref{sec:niwpd-proof} we prove Theorem~\ref{thm:base}.  This theorem gives the special case of Theorem~\ref{thm:main} where $\X = \{S\}$; this special case serves as a sort of base case for Theorem~\ref{thm:main}.  The idea of the proof of Theorem~\ref{thm:base} is to use the aformentioned construction of Bestvina--Bromberg--Fujiwara and Dahmani--Guirardel--Osin to translate the given action by WPD elements to an equivariant spinning action on a projection complex, and then to apply our result about actions on projection complexes, Theorem~\ref{thm:spin}.

In Section~\ref{sec:general-proof} we build on Theorem~\ref{thm:spin} and Theorem~\ref{thm:base} in order to prove Theorem~\ref{thm:main}.  The first step is to build a projection complex from the given $G$--overlapping family of subsurfaces $\X$; the vertices correspond to the subsurfaces of $S$ lying in the $G$--orbits of the elements of $\X$.  The $L$--spinning condition for this action is then ensured by Theorem~\ref{thm:base} and the orbit overlapping condition.  With this in place, we may apply Theorem~\ref{thm:spin} to obtain Theorem~\ref{thm:main}.  We note that Theorem~\ref{thm:spin} is applied twice overall, once in the proof of Theorem~\ref{thm:main} and once in the proof of Theorem~\ref{thm:base} (which is then used in the proof of Theorem~\ref{thm:main}).

In Section~\ref{sec:applications} we prove two applications of Theorem~\ref{thm:main}, namely, Theorems~\ref{thm:cong} and~\ref{thm:evenodd}.  Besides our theorem, both proofs use the Thurston construction for pseudo-Anosov mapping classes, and the theory of elementary closures of pseudo-Anosov mapping classes developed by McCarthy.   

\subsection{Acknowledgments} We would like to thank Mladen Bestvina, Ken Bromberg, Remi Coulon, Fran\c{c}ois Dahmani, Koji Fujiwara, Vincent Guirardel, Marissa Loving, Chris Leininger, Ashot Minasyan, Kasra Rafi, Alessandro Sisto, and an anonymous referee for helpful comments and conversations.  We thank the Mathematical Sciences Research Institute for hosting two of the authors during its 2016 program on Geometric Group Theory, where key parts of this work were completed.

The first author is partially supported by the Simons Foundation Grant No.~316383. The second author is supported by National Science Foundation Grant No.~DMS--1812021.  The third author is supported by National Science Foundation Grants No.~DMS--1057874 and No.~DMS--1811941.

\section{Classification of the RAAGs that arise from Theorem~\ref{thm:main}}\label{sec:examples}

We begin in Section~\ref{sec:overlap} by restricting the right-angled Artin groups that could arise from Theorem~\ref{thm:main} in the case when $G$ is normal in $\Mod(S)$; there are 15 possibilities.  In Section~\ref{sec:examples in mod} we show that all 15 possibilities do indeed arise for $\Mod(S_{g,p})$ with $g,p \geq 6$ and then we show that certain specific right-angled Artin groups arise in the closed case, where $p=0$.  Then in Section~\ref{sec:examples in torelli} we show that all 15 possibilities also arise in the cases of the pure braid group and the Torelli group.

\subsection{Overlapping subsurfaces}\label{sec:overlap}  While any subsurface $X \subseteq S$ is $G$--overlapping when $G = \Stab(X)$, when $G$ is a normal subgroup of $\Mod(S)$ there are topological conditions that must be satisfied by any $G$--overlapping subsurface.  These conditions then impose restrictions on the right-angled Artin groups that could arise from Theorem~\ref{thm:main} when $G$ is normal in $\Mod(S)$.

\begin{proposition}\label{prop:overlapping}
Suppose that $G \trianglelefteq \Mod(S)$, that $X$ is a $G$--overlapping subsurface, that $X_1$ and $X_2$ are arbitrary components of $X$, and that there is a partial pseudo-Anosov element or a power of a Dehn twist in $G$ with support contained in $X_1$.  Then every component of the boundary of $X_1$ either lies in $\partial S$ or is homotopic to a component of the boundary of $X_2$.    
\end{proposition}

\begin{proof}

Suppose for the sake of contradiction that some component $d$ of $\partial X_1$ is not homotopic to a component of the boundary of $S$ or $X_2$.  It follows that there is a curve $c$ that intersects $d$ but not $X_2$.  Let $f$ be a partial pseudo-Anosov mapping class or a Dehn twist with support contained in $X_1$.  By modifying $c$ if needed, we may further assume that $c$ intersects the support of $f$.  For any $m$ and $n$, the element $g = T_c^{-n} f^m T_c^{n}$ lies in $G$ as $G$ is normal in $\Mod(S)$.  For all choices of $m$ and $n$, we have that $gX_2=X_2$, as the support of $g$ is contained in the complement of $X_2$.  We will show for large $m$ and $n$ that $gX_1$ is not equal to $X_1$.  This will mean that $gX$ is not equal to $X$ and that $gX$ does not overlap $X$, contrary to the hypothesis that $X$ is $G$--overlapping.

To show that $gX_1$ is not equal to $X_1$ for large enough $m$ and $n$, we will show that  $g(d)$ intersects $d$ for large enough $m$ and $n$.  Applying $T_c^{n}$ to both curves $g(d)$ and $d$, this is equivalent to the statement that $f^mT_c^n(d)$ intersects $T_c^n(d)$.  For large $n$, $T_c^n(d)$ intersects the support of $f$ (this follows from \cite[Proposition 3.4]{primer}), and then for large $m$ it follows that $f^m(T_c^n(d))$ intersects $T_c^n(d)$ (in the case that $f$ is a power of a Dehn twist use the formula \cite[Proposition 3.4]{primer} again and if $f$ is a partial pseudo-Anosov element 
 use the fact due to Masur--Minsky \cite[Proposition~4.6]{ar:MM99} that $f$ has positive translation on the curve complex for the support of $f$ and the fact that the subsurface projection \cite{ar:MM00} of $T_c^n(d)$ to the support of $f$ is well defined).  This completes the proof.
\end{proof}

\begin{corollary}\label{co:axbx}
Suppose that $G \trianglelefteq \Mod(S)$, that $X$ is a $G$--overlapping subsurface of $S$, and that there is a partial pseudo-Anosov element or a power of a Dehn twist in $G$ with support contained in $X$.  Then, $0 \leq a_X \leq 1$ and $0 \leq b_X \leq 2$.    
\end{corollary}

\begin{proof}

Since a subsurface of $S$ does not have parallel annuli, it follows from Proposition~\ref{prop:overlapping} that $a_X \leq 1$.  Suppose $b_X \geq 2$ and let $X_1$ and $X_2$ be two non-annular components of $X$.  It follows from Proposition~\ref{prop:overlapping} that the complement in $S$ of $X_1 \cup X_2$ is the union of annuli.  Thus $b_X =2$, as desired.
\end{proof}

We say that a subsurface $X$ of $S$ is \emph{compatible} with a subsurface $Y$ if there is a homeomorphism $f$ of $S$ such that $f(X) \subseteq Y$; otherwise, we say that $X$ is \emph{incompatible} with $Y$.  We say that $X$ and $Y$ are \emph{mutually incompatible} if neither is compatible with the other.  Similarly, for a subgroup $G$ of $\Mod(S)$ we say that $X$ is \emph{$G$--compatible} with $Y$ if the above $f$ can be chosen so that its homotopy class lies in $G$; the terms \emph{$G$--incompatible} and \emph{mutually $G$--incompatible} are also defined analogously.

We say that a multicurve $C$ is a \emph{dividing set} if $S - C$ consists of two components and $S - C'$ is connected for any proper submulticurve $C' \subset C$.  In particular, any separating curve is a dividing set.  Given a dividing set $C$, we let $A$ be a closed regular neighborhood of $C$ and we let $L$ and $R$ be the two components of the complement of an open annular neighborhood of $A$.  

By Corollary~\ref{co:axbx} there are five possibilities for the pair $(a_X,b_X)$ for a $G$--overlapping subsurface $X$, where $G$ satisfies the hypotheses of the corollary.  Using  Proposition~\ref{prop:overlapping} we can explicitly describe $X$ in each of the five cases, as follows:

\begin{itemize}[itemsep=1.5ex]

\item[\fbox{(0,1)}] $X$ is a connected subsurface that is incompatible with its complementary region.  

\item[\fbox{(0,2)}] $X = L \cup R$ for a dividing set $C$ with $L$ and $R$ either homeomorphic or mutually incompatible.

\item[\fbox{(1,0)}] $X=A$ for a separating curve $C$ with $L$ and $R$ either homeomorphic or mutually incompatible.

\item[\fbox{(1,1)}] $X = L \cup A$ for a separating curve $C$ with $L$ and $R$ mutually incompatible.

\item[\fbox{(1,2)}] $X = L \cup A \cup R$ for a separating curve $C$ with $L$ and $R$ either homeomorphic or mutually incompatible.

\end{itemize}

Let $G$ be a normal subgroup of $\Mod(S)$ that satisfies the hypotheses of Corollary~\ref{co:axbx}.  By the corollary, the only RAAGs that may result from an application of Theorem~\ref{thm:main} are of the form
\[
\BigFreeProd{(a,b) \in A} \left( \BigFreeProd{\infty} \left((F_{\infty})^{b} \times \Z^{a}\right)\right)
\]
where $A$ is a subset of $\{0,1\} \times \{0,1,2\}$.  Since an infinite free product of infinite cyclic groups is isomorphic to an infinite free product of infinitely generated free groups, we may assume that $A$ does not contain, say, $(0,1)$.  In summary, there are 15 nontrivial isomorphism types of right-angled Artin subgroups of $G$ that may arise from an application of our Theorem~\ref{thm:main}, and they correspond to the 15 nonempty subsets of $\{(1,0),(1,1), (0,2),(1,2)\}$.

\subsection{Mapping class groups}\label{sec:examples in mod} In this section we prove that in certain mapping class groups, all 15 of the right-angled Artin groups from Section~\ref{sec:overlap} arise as normal subgroups.  Then we determine which of these 15 groups arise in the case that the surface is closed.

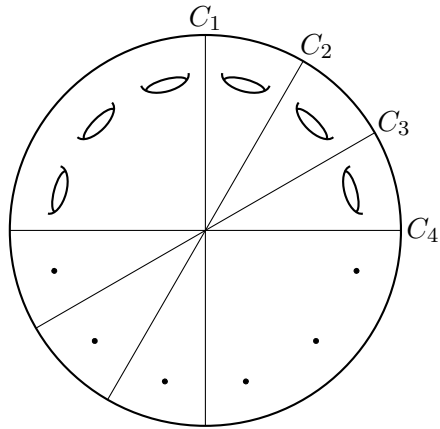
\begin{figure}[ht]
\centering
\begin{tikzpicture}[scale=.65]
\def\r{4}
\pgfmathsetmacro{\s}{0.8*\r}
\draw[thick] (0,0) circle [radius=\r];
\foreach \a in {0,1,2,3}{
	\draw[rotate=\a*30] (-1*\r,0) to[out=0,in=180] (\r,0);
}
\foreach \a in {15,45,75,105,135,165}{	
	\begin{scope}[rotate=\a]
	\begin{scope}[xshift=\s cm]
	\draw[thick] (0,0.5) arc (90:270:0.25cm and 0.5cm);
	\draw[thick] (-0.13,-0.43) arc (-90:90:0.13cm and 0.43cm);
	\end{scope}
	\begin{scope}[xshift=-1*\s cm]
	\filldraw (0,0) circle (0.05);
	\end{scope}
	\end{scope}
}
\node (c1) at (0,4.35) {$C_1$};
\node (c2) at (2.28,3.7789) {$C_2$};
\node (c3) at (3.8589,2.2) {$C_3$};
\node (c4) at (4.45,0) {$C_4$};
\end{tikzpicture}
\caption{The separating curves in Theorem~\ref{thm:examples}.}\label{fig:examples}
\end{figure}

\begin{theorem}\label{thm:examples}
Let $g,p \geq 6$.  For each of the 15 nonempty subsets $A$ of $\{(1,0),(1,1), (0,2),(1,2)\}$ there is a normal subgroup of $\Mod(S_{g,p})$ isomorphic to 
\[
\BigFreeProd{(a,b) \in A} \left( \BigFreeProd{\infty} \left((F_{\infty})^{b} \times \Z^{a}\right)\right).
\]
\end{theorem}

\begin{proof}

Let $C_1, C_2, C_3$ and $C_4$ be the four curves in $S_{g,p}$ as indicated in Figure~\ref{fig:examples} for the case of $S_{6,6}$; for $g > 6$ and $p > 6$ we may simply add more handles to the region with three handles and more marked points to the region with three marked points.  

Let $A_{i}$ be a closed annular neighborhood of $C_i$ and let $L_{i}$ and $R_{i}$ be the components of the complement of an open annular neighborhood of $A_{i}$ for $i=1,\ldots,4$.  By construction, any pair of subsurfaces in $\{L_{i} \mid i=1,\ldots,4\} \cup \{R_{i} \mid i = 1,\ldots 4\}$ is mutually incompatible with the possible exception of $L_{1}$ and $R_{1}$ (which are homeomorphic when $g=p=6$).  From this it follows that
\[ \X = \{A_{1}, L_{2} \cup A_{2}, L_{3} \cup R_{3}, L_{4} \cup A_{4} \cup R_{4} \} \]
is $\Mod(S_{g,p})$--overlapping. 

Applying Theorem~\ref{thm:main} to all possible subsets of $\X$ (and choosing $\F$ appropriately) we obtain the desired conclusion.

More specifically, each element $(a,b)$ of $\{(1,0),(1,1), (0,2),(1,2)\}$ corresponds to an element $X_{(a,b)} \in \X$ as follows:
\[
\begin{array}{rl}
(1,0)     &\leadsto \ A_1 \\
(1,1) &\leadsto \ L_2 \cup A_2 \\
(0,2) &\leadsto \ L_3 \cup R_3 \\ 
(1,2) &\leadsto\  L_4 \cup A_4 \cup R_4.
\end{array}
\]
We choose suitable mapping classes for each component of each such  $X_{(a,b)}$.  In each case, we may use the standard fact that a surface $S_{g',p'}$ containing an essential curve has pseudo-Anosov elements in its mapping class group; in particular each $L_i$ and $R_i$ that appears in $\X$ is the support of some partial pseudo-Anosov element.
\end{proof}

\begin{figure}[ht]
\centering
\begin{tikzpicture}[scale=.65]
\draw[thick,blue] (0,1.62) to[in=85,out=275] (0,-1.62);
\draw[thick,red] (-7,0.03) to[out=15,in=165] (-5.9,0.03) (-5.1,0.03) to[out=15,in=180] (-4.6,0.1) (-2.15,0.1) to[out=0,in=165] (-1.65,0.03);
\draw[thick,red] (-0.85,0.03) to[out=12,in=168] (0.85,0.03);
\draw[thick,red] (7,0.03) to[in=15,out=165] (5.9,0.03) (5.1,0.03) to[in=180,out=165] (4.6,0.1) (2.15,0.1) to[in=15,out=180] (1.65,0.03);
\draw[thick] (4,1.5) .. controls (8,1.25) and (8,-1.25) .. (4,-1.5);
\draw[thick] (-4,1.5) .. controls (-8,1.25) and (-8,-1.25) .. (-4,-1.5);
\draw[thick] (-4,1.5) .. controls (0,1.65) .. (4,1.5);
\draw[thick] (-4,-1.5) .. controls (0,-1.65) .. (4,-1.5);
\foreach \a in {-3.4,3.4} {
	\begin{scope}[xshift=\a cm]
	\draw[thick] (-2.7,-0.25) .. controls (-2,-1) and (-0.7,0) .. (0,-0.8);
	\draw[thick] (2.7,-0.25) .. controls (2,-1) and (0.7,0) .. (0,-0.8);
	\end{scope}
	\node (g) at (\a,-1.15) {$g/2$};
}
\foreach \a in {-5.5,-1.25,1.25,5.5}{
	\begin{scope}[xshift=\a cm,yshift=0.15cm,rotate=90]
	\begin{scope}
	\clip (-0.13,1) rectangle (-1,-1);
	\draw[thick,fill=white] (0,0.5) arc (90:270:0.25cm and 0.5cm);	
	\end{scope}
	\draw[thick,fill=white] (-0.13,-0.43) arc (-90:90:0.13cm and 0.43cm);	
	\draw[thick] (0,0.5) arc (90:270:0.25cm and 0.5cm);
	\end{scope}
}
\foreach \a in {2.9,3.4,3.9}{	
	\filldraw (\a,0.05cm) circle (0.025);
	\filldraw (-1*\a,0.05cm) circle (0.025);
}
\node[blue] (c1) at (0,2) {$C_1$};
\node[red] (c2) at (-7.5,0) {$C_2$};
\node (even) at (-10,0) {$g$ even};
\end{tikzpicture}

\vspace*{1cm}

\begin{tikzpicture}[scale=.65]
\draw[thick,blue] (1,1.62) to[in=85,out=275] (1,-1.62);
\draw[thick,red] (-7,0.03) to[out=15,in=165] (-5.9,0.03) (-5.1,0.03) to[out=15,in=180] (-4.6,0.1) (-1.15,0.1) to[out=0,in=165] (-.65,0.03);
\draw[thick,red] (0.15,0.03) to[out=12,in=168] (1.85,0.03);
\draw[thick,red] (7,0.03) to[in=15,out=165] (5.9,0.03) (5.1,0.03) to[in=180,out=165] (4.6,0.1) (3.15,0.1) to[in=15,out=180] (2.65,0.03);
\draw[thick] (4,1.5) .. controls (8,1.25) and (8,-1.25) .. (4,-1.5);
\draw[thick] (-4,1.5) .. controls (-8,1.25) and (-8,-1.25) .. (-4,-1.5);
\draw[thick] (-4,1.5) .. controls (0,1.65) .. (4,1.5);
\draw[thick] (-4,-1.5) .. controls (0,-1.65) .. (4,-1.5);
\begin{scope}[xshift=-3 cm]
\draw[thick] (-3.2,-0.25) .. controls (-2.5,-1) and (-0.7,0) .. (0,-0.8);
\draw[thick] (3.2,-0.25) .. controls (2.5,-1) and (0.7,0) .. (0,-0.8);
\end{scope}
\begin{scope}[xshift=3.9 cm]
\draw[thick] (-2.2,-0.25) .. controls (-1.4,-1) and (-0.7,0) .. (0,-0.8);
\draw[thick] (2.2,-0.25) .. controls (1.4,-1) and (0.7,0) .. (0,-0.8);
\end{scope}
\node (gL) at (-3,-1.12) {$(g+1)/2$};
\node (gR) at (3.9,-1.12) {$(g-1)/2$};
\foreach \a in {-5.5,-0.25,2.25,5.5}{
	\begin{scope}[xshift=\a cm,yshift=0.15cm,rotate=90]
	\begin{scope}
	\clip (-0.13,1) rectangle (-1,-1);
	\draw[thick,fill=white] (0,0.5) arc (90:270:0.25cm and 0.5cm);	
	\end{scope}
	\draw[thick,fill=white] (-0.13,-0.43) arc (-90:90:0.13cm and 0.43cm);	
	\draw[thick] (0,0.5) arc (90:270:0.25cm and 0.5cm);
	\end{scope}
}
\foreach \a in {2.5,3,3.5}{	
	\filldraw (-1*\a,0.05cm) circle (0.025);
}
\foreach \a in {3.5,3.9,4.3}{	
	\filldraw (\a,0.05cm) circle (0.025);
}
\node[blue] (c1) at (1,2) {$C_1$};
\node[red] (c2) at (-7.5,0) {$C_2$};
\node (odd) at (-10,0) {$g$ odd};
\end{tikzpicture}
\caption{The dividing sets for Theorem~\ref{thm:examples2}.}\label{fig:closed examples}
\end{figure}
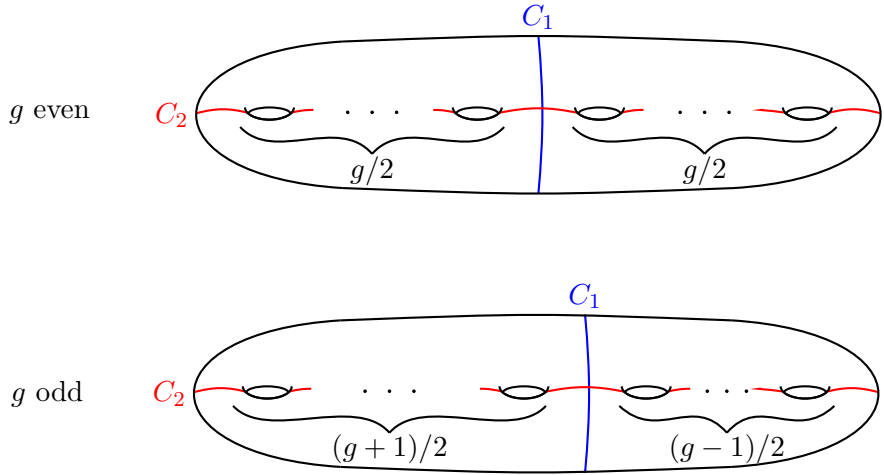

\begin{theorem}\label{thm:examples2}
Let $g \geq 2$.  
\begin{enumerate}
\item If $g$ is even, there are normal subgroups of $\Mod(S_g)$ isomorphic to 
\begin{align*}
F_\infty, \quad \BigFreeProd{\infty} (F_\infty \times F_\infty \times \Z) , \quad \BigFreeProd{\infty} (F_\infty \times F_\infty),
\end{align*}
and also the free product of the third with either of the first two if $g \geq 4$.
\item If $g$ is odd, there are normal subgroups of $\Mod(S_g)$ isomorphic to 
\[
F_\infty, \quad \FreeProdInf (F_\infty \times F_\infty),
\]
and also the free product of these.
\end{enumerate}
Moreover, these are the only normal right-angled Artin subgroups of $\Mod(S_g)$ that arise from the construction of Theorem~\ref{thm:main}.
\end{theorem}

\begin{proof}

We begin with the first two statements.  For the two cases ($g$ even and $g$ odd), let $C_1$ and $C_2$ be the multicurves indicated in the top and bottom of Figure~\ref{fig:closed examples}, respectively.  We define $A_i$, $L_i$, and $R_i$ as in the proof of Theorem~\ref{thm:examples}.  

The desired subgroups are constructed in the same way as in the proof of Theorem~\ref{thm:examples}.  For instance, when $g$ is even we may construct the free product of $\FreeProdInf (F_\infty \times F_\infty \times \Z)$ with $\FreeProdInf (F_\infty \times F_\infty)$ by taking 
\[ \X = \{A_{1} \cup L_{1} \cup R_{1}, L_{2} \cup R_{2} \}, \]
and we may construct the free product of $F_\infty$ with $\FreeProdInf (F_\infty \times F_\infty)$ by taking 
\[ \X = \{A_{1} , L_{2} \cup R_{2} \}. \]
Similarly, when $g$ is odd we may construct, for example, the free product $F_\infty$ with $\FreeProdInf (F_\infty \times F_\infty)$ by taking 
\[
\X = \{L_{1}, L_{2} \cup R_{2} \}.
\]
The other right-angled Artin groups from the statement of the theorem are obtained by removing elements (or components of elements) of the given sets $\X$.  

The last statement follows from a case-by-case analysis.  For example, the group $\FreeProdInf (F_\infty \times \Z)$ does not arise because there does not exist a separating curve in $S_g$ with the property that the two complementary regions are mutually incompatible.  For the other cases, the reasoning is similar.  
\end{proof}

\subsection{Pure braid groups and Torelli groups}\label{sec:examples in torelli}  The idea behind the examples in Section~\ref{sec:examples in mod} can be extended to create examples in both the pure braid group $\PB_n$ and the Torelli group $\I(S_g)$.

For the pure braid group, we require some setup.  A subsurface $X$ of $\D_n$ induces a partition of the marked points of $\D_n$: two marked points are in the same subset if and only if there is a path between them that does not cross the boundary of $X$.  We say that two partitions of a set are \emph{mutually compatible} if each component of one partition either contains or is contained in a component of the other partition.  It is a fact that two subsurfaces of $\D_n$ are mutually $\PB_n$--incompatible if the corresponding partitions of the marked points are mutually incompatible.

\begin{figure}[ht]
\centering
\begin{tikzpicture}[scale=.65]
\def\r{4.5}
\pgfmathsetmacro{\s}{0.8*\r}
\draw[thick] (0,0) circle [radius=\r];
\foreach \a in {0,1,2,3,4}{
	\filldraw[rotate=\a*72] (0.7*\r,0) circle (0.05);
}
\draw[thick] (0.15,0) .. controls (-0.25,6) and (4.1,3.25) .. (4.2,0);
\draw[thick] (0.15,0) .. controls (-0.25,-6) and (4.1,-3.25) .. (4.2,0);
\draw[thick,rotate=72,blue] (0.15,0) .. controls (-0.25,6) and (4.1,3.25) .. (4.2,0);
\draw[thick,rotate=72,blue] (0.15,0) .. controls (-0.25,-6) and (4.1,-3.25) .. (4.2,0);
\draw[thick,rotate=144,red] (0.15,0) .. controls (-0.25,6) and (4.1,3.25) .. (4.2,0);
\draw[thick,rotate=144,red] (0.15,0) .. controls (-0.25,-6) and (4.1,-3.25) .. (4.2,0);
\draw[thick,rotate=216,black!50!green] (0.15,0) .. controls (-0.25,6) and (4.1,3.25) .. (4.2,0);
\draw[thick,rotate=216,black!50!green] (0.15,0) .. controls (-0.25,-6) and (4.1,-3.25) .. (4.2,0);
\node (c1) at (5,0) {$C_1$};
\node[blue] (c2) at (1.5451,4.7553) {$C_2$};
\node[red] (c3) at (-4.0451,2.9389) {$C_3$};
\node[black!50!green] (c4) at (-4.0451,-2.9389) {$C_4$};
\end{tikzpicture}
\caption{The curves in the proof of Theorem~\ref{thm:examples pbn}.}\label{fig:examples pbn}
\end{figure}
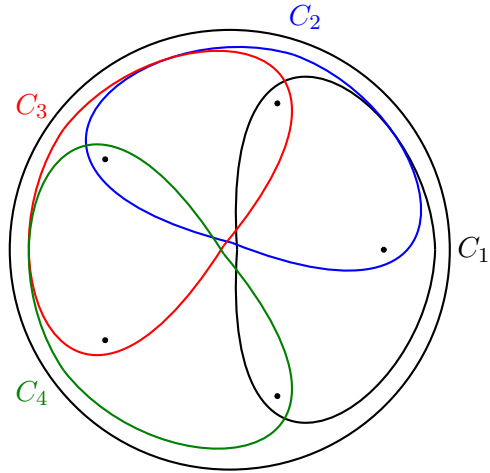

\begin{theorem}\label{thm:examples pbn}
Suppose $n \geq 5$.  For each of the 15 nonempty subsets $A$ of $\{(1,0),(1,1), (0,2),(1,2)\}$ there is a normal subgroup of $\PB_n$ isomorphic to 
\[
\BigFreeProd{(a,b) \in A} \left( \FreeProdInf \left((F_{\infty})^{b} \times \Z^{a}\right)\right).
\]
\end{theorem}

\begin{proof}

Let $C_1$, $C_2$, $C_3$, and $C_4$ be the four curves in $\D_n$ indicated in Figure~\ref{fig:examples pbn}.   The figure shows the case $n=5$; for $n > 5$ we may simply add marked points in the exterior of all four curves.  We define $A_i$, $L_i$, and $R_i$ as in the proof of Theorem~\ref{thm:examples}. 

Consider the following set of subsurfaces of $\D_n$:
\[
\{A_1,A_2,A_4\} \cup \{L_2,L_3,L_4\} \cup \{R_3, R_4\}.
\]
Any two elements of this set with distinct subscripts induce incompatible partitions of the set of marked points of $\D_n$.  As above, it follows that every pair of subsurfaces in the set
\[
\X = \{A_{1}, L_{2} \cup A_{2}, L_{3} \cup R_{3}, L_{4} \cup A_{4} \cup R_{4} \}
\]
is mutually $\PB_n$--incompatible.  From this it follows that $\X$ is $\PB_n$--overlapping.  Applying Theorem~\ref{thm:main} to all possible subsets of $\X$ (and choosing $\F$ appropriately) we obtain the desired conclusion (since each $L_i$ and $R_i$ that appears in $X$ is either a disk with 3 or more marked points or an annulus with two or more marked points, they all can be realized as the support of a partial pseudo-Anosov element of $\PB_n$).
\end{proof}

The situation for the Torelli groups is similar to that for the pure braid groups.  Let $X$ be a subsurface of $S_g$ that is either a separating annulus or a subsurface with connected boundary.  In either case $X$ induces a direct sum decomposition of $H_1(S_g;\mathbb{Q})$ into two symplectic subspaces.  In the annulus case, the two subspaces are the elements of $H_1(S_g;\mathbb{Q})$ with representatives on one side or the other of the annulus.  In the other case, the two subspaces are the elements of $H_1(S_g;\mathbb{Q})$ either represented in $X$ or its complement.  Similar to the case of the pure braid group, we say that two direct sum decompositions of $H_1(S_g;\mathbb{Q})$ are \emph{compatible} if each summand of one decomposition either contains or is contained in a summand of the other decomposition.

Below, the \emph{genus} of a separating curve $C$ in $S_g$ is the minimum genus of a complementary component.

\begin{figure}[ht]
\centering
\begin{tikzpicture}[scale=0.65]
\def\r{4.5}
\pgfmathsetmacro{\s}{0.3*\r}
\draw[thick] (0,0) circle [radius=\r];
\foreach \a in {0,...,3}{	
    \draw[thick,mc\a,rotate=90*\a] (0.05,1.35) .. controls (2,4) and (3.9,3) .. (4,0);
    \draw[thick,mc\a,rotate=90*\a] (-0.1,1.1) .. controls (-4.25,-4.5) and (4.1,-4.25) .. (4,0);
    \begin{scope}[rotate=\a*90]
    \begin{scope}[xshift=\s cm,rotate=-30]
    \draw[thick] (0,0.5) arc (90:270:0.25cm and 0.5cm);
    \draw[thick] (-0.13,-0.43) arc (-90:90:0.13cm and 0.43cm);
    \end{scope}
    \end{scope}
}
\begin{scope}[rotate=-10]\node[black] (c1) at (5,0) {$C_1$};\end{scope}
\begin{scope}[rotate=80]\node[blue] (c2) at (5,0) {$C_2$};\end{scope}
\begin{scope}[rotate=170]\node[red] (c3) at (5,0) {$C_3$};\end{scope}
\begin{scope}[rotate=260]\node[black!50!green] (c4) at (5,0) {$C_4$};\end{scope}
\end{tikzpicture}
\caption{The curves in the proof of Theorem~\ref{thm:examples torelli}.  Each $C_i$ is symmetric about the plane of the page, and only half of each curve is shown.}\label{fig:torelli examples}
\end{figure}
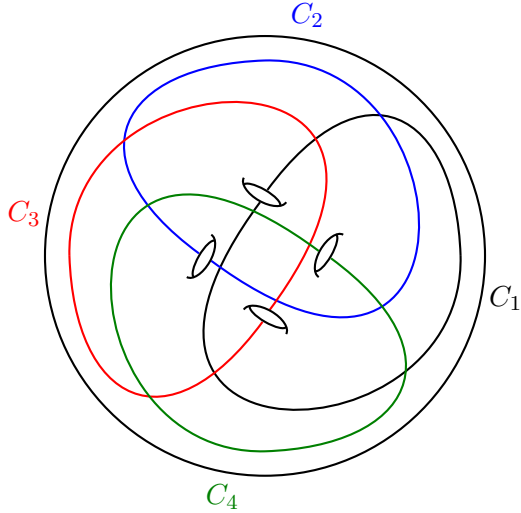

\begin{theorem}\label{thm:examples torelli}
Suppose $g \geq 4$.  For each of the 15 nonempty subsets $A$ of $\{(1,0),(1,1), (0,2),(1,2)\}$ there is a normal subgroup of $\I(S_g)$ isomorphic to 
\[
\BigFreeProd{(a,b) \in A} \left( \BigFreeProd{\infty} \left((F_{\infty})^{b} \times \Z^{a}\right)\right).
\]
\end{theorem}

\begin{proof}

Let $C_1$, $C_2$, $C_3$, and $C_4$ be four separating curves in $S_g$, each with genus at least 2, and so that the corresponding direct sum decompositions of $H_1(S_g;\mathbb{Q})$ are pairwise incompatible.  One such configuration for a surface of genus 4 is indicated in  Figure~\ref{fig:torelli examples}; for higher genus we may add handles to any complementary region.  We then define $A_i$, $L_i$, and $R_i$ as in the proof of Theorem~\ref{thm:examples}.  (For the surface of genus 4, any 4 curves of genus 2 giving distinct direct sum decompositions of $H_1(S_g;\mathbb{Q})$ will suffice, and any such configuration can be extended to higher genus by adding extra handles to a single complementary region).

Consider the following set of subsurfaces of $S_g$:
\[
\{A_1,A_2,A_4\} \cup \{L_2,L_3,L_4\} \cup \{R_3, R_4\}.
\]
The decomposition of $H_1(S_g;\mathbb{Q})$ induced by a separating curve in $S_g$ is the same as the decomposition induced by either of its complementary regions.  Thus, any two elements of the above set with distinct subscripts induce incompatible direct sum decompositions of $H_1(S_g;\mathbb{Q})$.  Because $\I(S_g)$ acts trivially on $H_1(S_g;\mathbb{Q})$, it follows that every pair of subsurfaces in the set
\[
\X = \{A_{1}, L_{2} \cup A_{2}, L_{3} \cup R_{3}, L_{4} \cup A_{4} \cup R_{4} \}
\]
is mutually $\I(S_g)$--incompatible.  It follows that $\X$ is $\I(S_g)$--overlapping.  Applying Theorem~\ref{thm:main} to all possible subsets of $\X$, and choosing $\F$ appropriately, we obtain the desired conclusion (we may apply, for example, the Thurston construction of pseudo-Anosov mapping classes to produce partial pseudo-Anosov elements of $\I(S_g)$ with support in any $L_i$ or $R_i$).
\end{proof}

\section{Projection complexes}\label{sec:back}

In Section~\ref{subsec:def and prop} we recall the definition of a projection complex and explain how any tree may be viewed as a projection complex.  Then in Section~\ref{sec:eg} we recall the theory of WPD elements and explain the construction of a projection complex from a collection of WPD elements.  Finally, in Section~\ref{sec:first} we explain how to build a projection complex from the curve complexes of orbit overlapping subsurfaces of a given surface; our new contribution, Proposition~\ref{prop:disconnected subsurfaces}, is a version for disconnected subsurfaces.


\subsection{The definition}\label{subsec:def and prop}

Let $\Y$ be a set and let $\theta \geq 0$ be a constant.   Assume that for each $y \in \Y$ there is a function
\[
d_{y} \from (\Y \setminus \{y\}) \times (\Y \setminus \{y\}) \to \R_{\geq 0}
\]
with the following properties.

\medskip

\begin{enumerate}[leftmargin=0.5cm,itemsep=.5em]
\item[] {\it Symmetry:} \ $d_{y}(x,z) = d_{y}(z,x)$ for all $x,y,z \in \Y$
\item[] {\it Triangle inequality:} \ $d_{y}(x,z) + d_{y}(z,w) \geq d_{y}(x,w)$ for all $x,y,z,w \in \Y$
\item[] {\it Inequality on triples:} \ $\min \{ d_{y}(x,z), d_{z}(x,y)  \} \leq \theta$ for all $x,y,z \in \Y$
\item[] {\it Finiteness:} \ $\#\{ y \in \Y \mid d_{y}(x,z) > \theta  \}$ is finite  for all $x,z \in \Y$
\end{enumerate}

\medskip

\noindent These conditions are known as the \emph{projection complex axioms}.  When we say that a set $\Y$ and a collection of functions $\{d_y\}_{y \in \Y}$ as above satisfy the projection complex axioms the constant $\theta$ is implicit.  

For a given $K \geq  0$, we will define a graph $\P_K(\Y)$ with vertices corresponding to the elements in $\Y$.  To define the edges, we require the notion of modified distance functions.

Given the functions $\{d_y\}$, Bestvina--Bromberg--Fujiwara \cite{ar:BBF15} constructed another collection of functions $\{d_{y}'\}_{y \in \Y}$, where each $d_y'$ shares the same domain and target as $d_y$.  Because the definition of the $d_y'$ is technical and because we do not use the definition in this paper, we do not state it here.  Bestvina--Bromberg--Fujiwara~\cite[Theorem~3.3B]{ar:BBF15} showed that the modified functions are coarsely equivalent to the original functions: for $x \neq y \neq z \in \Y$, $d'_{y}(x,z) \leq d_{y}(x,z) \leq d'_{y}(x,z) + 2\theta$.

Fix $K \geq 0$.  Then two vertices $x,z$ of $\P_{K}(\Y)$ are connected by an edge if $d'_{y}(x,z) \leq K$ for all $y \in \Y - \{x,z\}$.  Let $d$ denote the resulting path metric on $\P_{K}(\Y)$.  

Bestvina--Bromberg--Fujiwara showed that for $K$ large enough relative to $\theta$, there are constants $\Ke$, $\Kp$, and $\Kg$,  so that the following properties hold (see \cite[Proposition 3.14, and Lemma 3.18]{ar:BBF15}):

\medskip \noindent {\bf Bounded edge image.} If $x\neq y \neq z$ are vertices of $\P_K(\Y)$ and $d(x,z)=1$, then $d_{y}(x,z) \leq \Ke$.  

\medskip \noindent {\bf Bounded path image.}  If a path in $\P_K(\Y)$ connects vertices $x$ to $z$ without passing through the 2--neighborhood of the vertex $y$, then $d_y(x,z) \leq \Kp$. 

\medskip \noindent {\bf Bounded geodesic image.} If a geodesic in $\P_K(\Y)$ connects vertices $x$ to $z$ without passing through the vertex $y$, then $d_y(x,z) \leq \Kg$. \medskip

(The bounded edge image property follows from the definition of the edges of $\P_K(\Y)$, with $\Ke = K+2\theta$.)  If $K$ is large enough so that the graph $\P_K(\Y)$ satisfies the bounded edge, path, and geodesic properties for some $\Ke$, $\Kp$, and $\Kg$, then we say that $\P_K(\Y)$ is a \emph{projection complex}.  

For a projection complex, we refer to $\Ke$, $\Kp$, and $\Kg$ as the \emph{edge constant}, the \emph{path constant}, and the \emph{geodesic constant}, respectively.

We note that our terminology is not standard; in the papers by Bestvina--Bromberg--Fujiwara~\cite{ar:BBF15} and Bestvina--Bromberg--Fujiwara--Sisto~\cite{un:BBFS}, every $\P_K(\Y)$ is called a projection complex.  

\p{Group actions on projection complexes.} We say that a group $G$ acts on a projection complex $\P_K(\Y)$ if $G$ acts on the set $\Y$ in such a way that the associated distance functions $d_y$ are $G$--invariant, i.e., $d_{gy}(gx,gz) = d_{y}(x,z)$.  We note that if the original distance functions $d_y$ are $G$--invariant, then the modified distance functions are $G$--invariant as well, and so the action of $G$ on $\Y$ in particular extends to an action of $G$ on the graph $\P_{K}(\Y)$ by simplicial automorphisms.  

\p{Trees are projection complexes.} As an illustration of the definition of a projection complex, we now explain how an arbitrary simplicial tree can be viewed as a projection complex.  Let $T$ be a simplicial tree and $\Y$ the set of vertices in $T$.  We set:
\begin{equation*}\label{eq:tree distance}
d_{y}(x,z) = \begin{cases}
1 & \mbox{ if $y$ is on the geodesic from $x$ to $z$}, \\
0 & \mbox{ otherwise}.
\end{cases}
\end{equation*}
These functions satisfy the projection complex axioms for any $\theta \geq 0$.  Notice that the inequality on triples axiom is merely the fact that for any triple of points in a tree, at most one is on the geodesic between the other two.  For this example, the modified distance functions are the same as the original distance functions and moreover for any $0 < K < 1$ we see that $\P_{K}(\Y)$ is isomorphic to $T$. On the other hand, if $K \geq 1$ then $\P_{K}(\Y)$ is a complete graph.


\subsection{Projection complexes from WPD elements}\label{sec:eg}

In this section we give the first important general construction of a projection complex.  We begin by briefly describing two motivating examples.  Suppose we have either
\begin{enumerate}
    \item a collection of elements of a free group acting on the Cayley graph of the free group, or
    \item a collection of elements of a surface group acting on the hyperbolic plane.
\end{enumerate}
In either case, there is a projection complex where the vertices are the orbits of the axes of the given elements and the distance functions are given as follows: if $x$, $y$, and $z$ are axes for the given elements, then $d_y(x,z)$ is the diameter of the union of the nearest-point projections of $x$ and $z$ to $y$.  It is an exercise in either graph theory or hyperbolic geometry to show that these functions give rise to a projection complex.   

In the rest of this section, we describe a general setup for constructing projection complexes, inspired by the above two examples.  The end result is Proposition~\ref{prop:many}.  In order to state it we require a number of definitions.  In the remainder of this section, let $X$ be a hyperbolic metric space and let $G$ be a group acting on $X$.

\p{Distance functions from nearest-point projection.}  For a subset $Z$ of $X$ and a point $p \in X$, the \emph{projection of $p$ to $Z$} is the (possibly empty) subset
\begin{equation*}
\pi_{Z}(x) = \{ z \in Z \mid d(x,z) \leq d(x,z') \mbox{ for all } z' \in Z \}.
\end{equation*}
This defines a function $\pi_{Z} \from X \to \wp(Z)$.

Let $\Y$ be a collection of subsets of $X$ and for each $a$ in $\Y$, define $\pi_a \from X \to \wp(a)$ as above.  If $\pi_a(b)$ is bounded and nonempty for all pairs $a,b \in \Y$, we may define $d_a \from \Y\setminus \{a\} \times \Y\setminus\{a\} \to \R$ by
\[ 
d_a(b,c) = \diam(\pi_a(b) \cup \pi_a(c)),
\]
Then the family $\{d_a\}_{a \in \Y}$ satisfies the first two projection complex axioms.  Our goal in what follows is to describe conditions under which the other two axioms are also satisfied.  

One easy special case is where $X$ is a simplicial tree and $\Y$ is any collection of disjoint geodesics.  To verify the last two axioms, the key observation is that the geodesics $y$ with nonzero $d_y(x,z)$ are precisely those sharing an edge with the shortest path between $x$ and $z$.

\p{Translation length and hyperbolic elements.}   Let $f \in G$.  The \emph{translation length} of $f$ is the limit
\begin{equation*}
\tau(f) = \lim_{n \to \infty} \frac{d(p,f^{n}p)}{n}.
\end{equation*}
for some $p \in X$; this limit is independent of the point $p$.  We say $f$ is \emph{hyperbolic} if $\tau(f) > 0$.  

\p{Weak proper discontinuity.}
We say $f \in G$ is a \emph{WPD element} if it is hyperbolic and for all $D \geq 0$ and $p \in X$ there exists $M \geq 0$ such that the set
\[ \{ g \in G \mid d\left(p,gp\right) \leq D \mbox{ and } d\left(f^{M}p,g\left(f^{M} p\right)\right) \leq D \} \]
is finite.  The WPD elements for the $\Mod(S)$ action on $\C(S)$ are precisely the pseudo-Anosov mapping classes~\cite{ar:BF02}.

\p{Elementary closure.} Suppose $f \in G$ is a WPD element and let $\calO$ be the orbit of some point $X$ under $\grp{f}$.  Define the \emph{elementary closure}, $\EC(f)$, to be the subgroup of $g \in G$ such that there exists a $\Delta \geq 0$ with $g\calO$ contained in the $\Delta$--neighborhood of $\calO$.  The group $\EC(f)$ is well-defined independent of the orbit.  Bestvina--Bromberg--Fujiwara proved \cite[Proposition 4.7]{un:BBF-v1} that for such an $f \in G$, there is a short exact sequence
\begin{equation*}\label{eq:elementary closure}
1 \to \EC_{0}(f) \to \EC(f) \to \Gamma \to 1,
\end{equation*}
where $\Gamma$ is isomorphic to either $\Z$ or $\Z_{2} \ast \Z_{2}$ and $\EC_{0}(f)$ is finite.  The set $\EC(f) \cdot \calO$ is called a \emph{quasi-axis bundle} for $f$.  It follows from the definitions that the stabilizer of any quasi-axis bundle is $\EC(f)$.

As $X$ is hyperbolic, any hyperbolic isometry acts with north--south dynamics on the boundary of $X$.  It follows readily that for a WPD element $f \in G$, the group $\EC(f)$ is the stabilizer in $G$ of the pair of fixed points in $\partial X$ for $f$.  In particular, if $f \in \Mod(S)$ is pseudo-Anosov and $X$ is the curve complex $\C(S)$, then $\EC(f)$ is the stabilizer in $G$ of the pair of projective measured foliations on $S$ fixed by $f$, as defined in the introduction.

\p{Normal independence.} We say that WPD elements $f_{1}$ and $f_{2}$ in $G$ are \emph{independent} if $\EC(f_{1}) \cap \EC(f_{2})$ is finite.  This is equivalent to requiring that the fixed points sets for $f_1$ and $f_2$ in $\partial X$ are disjoint.  We further say that $f_{1}$ and $f_{2}$ are \emph{normally independent} if every conjugate of $f_{1}$ is independent of $f_{2}$.  We say that a collection of WPD elements is \emph{\textup{(}normally\textup{)} independent} if they are pairwise (normally) independent.

For Propositon~\ref{prop:many}, we will use the easy observation that normally independent elements cannot have conjugate powers.  In fact, one can prove that two WPD elements of a group acting on a hyperbolic space are normally independent if and only if they fail to have conjugate nontrivial powers.

\bigskip

The following result of Dahmani--Guirardel--Osin \cite{ar:DGO17} will be used to construct a projection complex in Section~\ref{sec:niwpd-proof} (for the action of the mapping class group on the curve complex).

\begin{proposition}[Dahmani--Guirardel--Osin]\label{prop:many}
Let $G$ be a group and $X$ a hyperbolic metric space on which $G$ acts.   Let $\{f_{1},\ldots,f_{m}\}$ be a normally independent family of WPD elements and for each $f_i$ let $\beta_i \subseteq X$ be a quasi-axis bundle. Then the set $\Y= G \cdot \{\beta_1,\ldots,\beta_m\}$ together with the distance functions $\{d_\beta\}_{\beta \in \Y}$ satisfy the projection complex axioms.
\end{proposition}

If $\P$ is a projection complex arising from Proposition~\ref{prop:many}, and if the translation distances of the $f_i$ are all bounded below by $L$, then the collection $\{\grp{gf_i g^{-1}} \mid g \in G, \, 1 \leq i \leq m \}$ is an equivariant $L$--spinning family of subgroups.   We will use this fact in the proof of Theorem~\ref{thm:base} in Section~\ref{sec:niwpd-proof}.

Dahmani--Guirardel--Osin did not state Proposition~\ref{prop:many} explicitly, but it follows from their work as we now explain.  Proposition~\ref{prop:many} has the same hypothesis as Theorem 6.8 in their paper, whose proof hinges on showing that the groups $\EC(f_i)$ are geometrically separated, as defined in that paper.  Lemma 4.47 in their paper proves, in particular, that the $G$--translates of the quasi-axis bundles (i.e.~the orbits of the geometrically separated subgroups) and the associated distance functions satisfy the  projection complex axioms.  In their paper, the quasi-axes bundles $\beta_i$ are all orbits of a common basepoint, but standard coarse geometry for hyperbolic spaces can be used to extend Lemma 4.47 to arbitrarily chosen quasi-axis bundles, up to changing the constants of the projection complex.

Proposition~\ref{prop:many} actually holds with the assumption that $X$ is a hyperbolic space replaced by the requirement that the WPD elements have strongly contracting axes.  The argument in the case of a single element appears in the first version of a paper by Bestvina--Bromberg--Fujiwara \cite{un:BBF-v1}, and we prove the general case in a forthcoming paper \cite{ar:CMMexpos}.


\subsection{Projection complexes from subsurface projections}\label{sec:first}

We recall the Bestvina--Bromberg--Fujiwara construction of the projection complex associated to the curve complexes of a collection of orbit-overlapping connected subsurfaces of a given surface.  We then give our generalization to the case of disconnected subsurfaces.

\p{Projection complexes from connected $G$--overlapping subsurfaces.} Given a nonannular connected subsurface $X \subseteq S$, and a curve $\gamma$ in the curve complex $\C(S)$, the \emph{projection} of $\gamma$ to $X$ is the subset $\pi_{X}(\gamma) \subset \C(X)$ of curves that have a representative disjoint from some some arc in $X \cap \gamma$.  When $X \subseteq S$ is an annulus, the definitions of $\C(X)$ and $\pi_{X} \from \C(S) \to \wp(\C(X))$ are more complicated; see~\cite{ar:MM00} for details.  In either case, when $\pi_{X}(\gamma)$ is nonempty, the diameter of $\pi_{X}(\gamma)$ is at most 2~\cite[Lemma~2.3]{ar:MM00}.  Moreover, the set $\pi_{X}(\gamma)$ is nonempty   whenever $\gamma$ essentially intersects $X$.

If $\Y$ is a $\Mod(S)$--invariant family of isotopy classes of connected subsurfaces of $S$ such that $\pi_Y(\partial X)$ is nonempty for all distinct $X,Y \in \Y$, then subsurface projection defines a function $d_Y \from \Y - \{Y\} \times \Y - \{Y\} \to \R_{\geq 0}$ by:
\begin{equation*}\label{eq:subsurface distance}
d_{Y}(X,Z) = \diam_{\C(Y)}(\pi_{Y}(\partial X) \cup \pi_{Y}(\partial Z)).
\end{equation*} 
These functions clearly satisfy the first two projection complex axioms.  That the inequality on triples axiom is satisfied for $\theta = 10$ is known as the Behrstock inequality~\cite{ar:Behrstock06,ar:Mangahas10}.   Bestvina--Bromberg--Fujiwara~\cite[Lemma~5.3]{ar:BBF15} show that the finiteness axiom holds for $\theta = 3$.

\p{Projection complexes from arbitrary $G$--overlapping subsurfaces.}  Let $G$ be a subgroup of $\Mod(S)$ and let $\X$ be an $G$--overlapping family of subsurfaces of $S$.  Set $\Y = G \cdot \X$.  Given $Y \in \Y$ we define a function $\dd_{Y} \from \Y - \{Y\} \times \Y - \{Y\} \to \R_{\geq 0}$ by:
\begin{equation*}\label{eq:disconnected projection sum}
\dd_{Y}(X,Z) = \sum_{Y' \in \pi_{0}(Y)} d_{Y'}(X,Z).
\end{equation*}
where $d_{Y'}$ is the subsurface projection distance defined above.  As $\X$ is an $G$--overlapping family, each term in the summand is defined.  The following proposition gives us the projection complex used for Theorem~\ref{thm:main}.

\begin{proposition}\label{prop:disconnected subsurfaces}
Suppose $G$ is a subgroup of $\Mod(S)$ and $\X$ is a $G$--overlapping family of subsurfaces.  Then the set  $\Y = G \cdot \X$ together with the distance functions $\{ \dd_Y \}_{Y \in \Y}$ satisfy the projection complex axioms.
\end{proposition}

\begin{proof}The first two projection complex axioms hold for the functions $\dd_{Y}$, since they hold for the summands.  

Let $C$ denote the maximum number of components of a subsurface in $\Y$.
Suppose $\dd_{Y}(X,Z) > 12C$.  Thus $d_{Y'}(X,Z) > 12$ for some $Y' \in \pi_{0}(Y)$ and hence $d_{Y'}(X',Z') > 10$ for each $X' \in \pi_0(X)$ and $Z' \in \pi_0(Z)$.  Thus by the Behrstock inequality, we have $d_{X'}(Y,Z) \leq 10$ for any component $X' \in \pi_{0}(X)$.  Hence $\dd_{X}(Y,Z) \leq 10C$.  This shows that the inequality on triples axiom holds for $\{\dd_Y\}_{Y \in \Y}$ using $\theta = 12C$.

Finally, as $\#\{ Y \mid d_{Y}(X,Z) > 3 \} < \infty$ for any connected subsurfaces $X,Z \subseteq S$, we see the same is true for $\{ Y \in \Y \mid \dd_{Y}(X,Z) > 5C \}$ for any $X,Z \in \Y$.  Indeed, fix $X,Z \in \Y$.  If $\dd_{Y}(X,Z) > 5C$, then $d_{Y'}(X,Z) > 5$ for some $Y' \in \pi_0(Y)$.  Hence $d_{Y'}(X',Z') > 3$ for any $X' \in \pi_0(X)$ and $Z' \in \pi_0(Z)$.  Thus there are only finitely many possibilities for the subsurface $Y'$.  However, as the subsurfaces in $\Y$ are either equal or overlapping, such a subsurface $Y'$ can be a component for only a single subsurface in $\Y$.  Thus the finiteness axiom holds and the proposition is proved.  
\end{proof}

\section{Windmills and waypoints in projection complexes}
\label{sec:windmillsprojection}

Here we develop our theory of windmills for group actions on projection complexes.  In Section~\ref{sec:proof} we use this to prove Theorem~\ref{thm:spin}, which states that if we have a group $G$ acting on a projection complex $\P$ and if we have an equivariant $L$--spinning family $\{R_v\}$ of subgroups of $G$, then when $L$ is large, the subgroup of $G$ generated by all $R_v$ is isomorphic to the free product of some of the $R_v$.  

We begin with the basic definitions regarding windmills, which allow us to rephrase Theorem~\ref{thm:spin} as  Proposition~\ref{prop:rephrase} below. Then we introduce the notions of pivot points and waypoints, which gives a further reduction of Proposition~\ref{prop:rephrase}, given in Proposition~\ref{prop:pivotway}.  


\subsection{Windmill data}
\label{sec:data}

Given an action of a group $G$ on a projection complex $\P$ with an equivariant family of subgroups $\{R_v\}$ of $G$, we can inductively define a sequence of subgraphs $W_i$ of $\P$, a sequence of subsets $\calO_i$ of the set of vertices of $\P$, and a sequence of subgroups $H_i$ of $G$ as follows.  (In fact, the constructions of this section apply to group actions on arbitrary graphs.)

Let $v_0$ be some base point for $\P$.  To begin the inductive definitions at $i=0$, we define:
\begin{itemize}
\item $H_0 = R_{v_0}$ and
\item $W_0 = \calO_0 =  \{v_0\}$.
\end{itemize}
For $i \geq 1$, we denote by $N_i$ the 1--neighborhood of $W_{i-1}$, we denote by $L_i$ the vertices of $N_i \setminus W_{i-1}$, and we define:
\begin{itemize}
\item $H_i = \grp{R_v \mid v \in N_{i}}$,
\item $W_i = H_i \cdot N_i$, and
\item $\calO_i =$ a set of orbit representatives for the action of $H_{i-1}$ on $L_i$.
\end{itemize}
The letter W here stands for ``windmill'' and we refer to the set
\[
\{(H_i,W_i,\calO_i)\}_{i=0}^{\infty}
\]
as a set of \emph{windmill data} for the equivariant family $\{R_v\}$.

We can see inductively that each $W_i$ is connected.  The keys to the inductive step are that $H_i$ acts on $W_i$ and that each generator $g$ of $H_i$ has the property that $\left( g  N_i  \right) \cap N_i \neq \emptyset$.

Given the above windmill data, we may further define both a sequence of groups and a sequence of homomorphisms, for $i \geq 0$:
\begin{itemize}
\item $F_i = \displaystyle\left(\BigFreeProd{v \in \calO_i}R_v\right) * F_{i-1}$ and
\item $\rho_i \from F_i \to H_i$, the natural map.
\end{itemize}
In the definition of $F_0$ we treat $F_{-1}$ as the trivial subgroup, and so $F_0 = R_{v_0}$.  Note that, by definition, each $\rho_i$ is injective on each $R_v \subseteq F_i$.  

We refer to the collection $\{(F_i,\rho_i)\}_{i=0}^\infty$ as the \emph{free product data} associated to the windmill data $\{(H_i,W_i,\calO_i)\}_{i=0}^{\infty}$.  

We denote the direct limit of the $F_i$ by $F$:
\[
F = \lim_{\longrightarrow} F_i.
\]
There is a natural partition of $F$ into sets $\{F^{(i)}\}_{i=-1}^\infty$ defined by
\[
F^{(i)} = F_i \setminus F_{i-1},
\]
where we treat $F_{-2}$ as the empty set, and so $F^{(-1)} = \{\textrm{id}\}$ and $F^{(0)} = R_{v_0} \setminus \{\textrm{id}\}$.  We refer to each $F^{(i)}$ as the $i$th \emph{level} of $F$.

The subgroup $H$ of $G$ generated by the $R_v$ is the direct limit of the $H_i$.   With this setup in hand we may rephrase Theorem~\ref{thm:spin} in the following way.

\begin{proposition}
\label{prop:rephrase}
Let $\P$ be a projection complex, $G$ a group that acts on $\P$, and $\{R_v\}$ an equivariant family of subgroups.  Choose windmill data $\{(H_i,W_i,\calO_i)\}_{i=0}^{\infty}$.  Then the conclusion of Theorem~\ref{thm:spin} holds if each $\rho_i$ is injective.
\end{proposition}

\begin{proof}

If $v$ and $w$ are two vertices of $\P$ in the same $H$--orbit, then by the equivariance condition, we have $R_v \cong R_w$.  Thus, we may check that the conclusion of Theorem~\ref{thm:spin} holds for any particular choice $\calO$ of orbit representatives.  We will check that it holds for $\calO = \cup \calO_i$.  

First we verify that $\cup \calO_i$ is indeed a set of orbit representatives.  The equivariance condition means two vertices $v, w$ in the same orbit have conjugate corresponding subgroups $R_v$ and $R_w$.  On the other hand, if $v,w \in \cup\calO_i$, then $R_v$ and $R_w$ are free factors in the domain $F_i$ of $\rho_i$, and distinct free factors of a free product have trivially intersecting conjugates.  Thus injectivity of the $\rho_i$ implies that distinct vertices of $\cup \calO_i$ represent distinct $H$--orbit representatives.  We must also show that every $H$--orbit is represented in $\cup \calO_i$.  Given $O$ an $H$--orbit of vertices in $\P$, there exists a minimal $i$ such that $O \cap N_i$ is not empty.  By minimality of $i$, if $w$ is in $O \cap N_i$, then it is not contained in $W_{i-1}=H_{i-1}\cdot N_{i-1}$.  In particular $w \in L_i$ and thus has an orbit representative in $\calO_i$.

The $\rho_i$ induce a natural surjective map $\rho \from F \to H$.  The statement that $\grp{R_v}$ is isomorphic to the free product
\[
\BigFreeProd{v\in \cup \calO_i} R_v
\]
is equivalent to the statement that $\rho$ is injective, which is in turn equivalent to the statement that each $\rho_i$ is injective.  This is true by hypothesis, and so the proof is complete.
\end{proof}

\subsection{Pivot points and waypoints}\label{sec:pivot-way}

We continue with the notation of the previous section.  To each element of $F$ at a given level $i\geq 0$, we associate a subset of the vertices of $\P$, as follows.  Each $h \in F^{(i)}$ has a \emph{syllable decomposition} $h_1 \cdots h_n$ with respect to the free product decomposition used to define $F_i$.  Specifically, each syllable $h_j$ is either a nontrivial element of $F_{i-1}$ or a nontrivial element of $R_{v_j}$ with $v_j \in \calO_i$, and also the following property is satisfied: no two consecutive syllables are of the first type and consecutive syllables $h_j$ and $h_{j+1}$ of the second type have distinct corresponding fixed vertices $v_j$ and $v_{j+1}$.  We refer to $n$ as the \emph{syllable length} of $h$.

Let $i \geq 1$ and fix some $h \in F^{(i)}$ with syllable decomposition $h=h_1\cdots h_n$.  For $j \in \{1,\dots, n\}$ with $h_j \notin F_{i-1}$ and with corresponding fixed vertex $v_j$ we define a vertex $w_j$ of $\P$ as follows:
\[
w_j = h_1\dots h_{j-1}v_j.
\]
Note that $v_j$ and $w_j$ are not defined for $h_j \in F_{i-1}$.  Let $\W(h)$ be the ordered list of points $w_j$, and call these the \emph{pivot points} for $h$. For $h \in F_0$ we define $\W(h)$ to be empty.

Each group $F_i$ acts on $\P$ via $\rho_i$.  We say that a vertex $v$ of $\P$ is a \emph{waypoint} for $h \in F_i$ if it is not equal to $v_0$ or $hv_0$ and
every geodesic from $v_0$ to $hv_0$ passes through $v$.  If at least one pivot point for $h$ is a waypoint, it follows that $\rho_i(h)$ is nontrivial.  Recall that Proposition~\ref{prop:rephrase} reduces Theorem~\ref{thm:spin} to proving injectivity of $\rho_i$ for all $i$.  Hence we have further reduced our task to proving that every nontrivial $h \in F^{(i)}$ has a waypoint, for all $i \geq 1$.  We record this fact in the following proposition.

\begin{proposition}\label{prop:pivotway}
Let $\P$ be a projection complex, $G$ a group that acts on $\P$, and $\{R_v\}$ an equivariant family of subgroups.  Choose windmill data $\{(H_i,W_i,\calO_i)\}_{i=0}^{\infty}$ and let $\{(F_i,\rho_i)\}_{i=0}^\infty$ be the associated free product data.  If for each $i \geq 1$ and each $h \in F^{(i)}$ at least one element of $\W(h)$ is a waypoint for $h$, then the conclusion of Theorem~\ref{thm:spin} holds.
\end{proposition}

\section{Windmills in trees}\label{sec:trees}

In this section, we show how the windmill machinery developed in Section~\ref{sec:windmillsprojection} is used in a special case, namely, the case of a group acting on a tree.  Let $G$ be a group acting on a tree $\T$, with $\{R_v\}$ an equivariant spinning family of subgroups.  The spinning condition takes a particularly simple form in the case of a tree, as follows.
\begin{itemize}
\item \emph{Spinning condition for trees:} For any vertex $v$ of $\T$, any incident edge $e$ of $\T$, and any nontrivial $h \in R_v$ we have
\[
h(e) \neq e.
\]
\end{itemize}
We have the following special case of Theorem~\ref{thm:spin}.

\begin{theorem}
\label{thm:treecase}
Let $\T$ be a tree and let $G$ be a group acting on $\T$.  If $\{R_v\}$ is an equivariant spinning family of subgroups of $G$ with windmill data $\{(H_i,W_i,\calO_i)\}_{i=0}^{\infty}$ then the subgroup $H$ of $G$ generated by the $R_v$ is isomorphic to the free product
\[
\BigFreeProd{v\in \cup \calO_i} R_v.
\]
\end{theorem}

Theorem~\ref{thm:treecase} is a recasting of a standard fact from the Bass--Serre theory of group actions on trees.  Since the $R_v$ are vertex stabilizer subgroups, it follows from Bass--Serre theory that the quotient $\T/H$ is a tree and hence that $H$ is isomorphic to a free product
\[
\FreeProd{} R_{v_j}
\]
where $\{v_j\}$ is a set of orbit representatives for the action of $H$ on the set of vertices of $\T$ (see  \cite[Section 5.1]{Serre}).  Because $\T$ is a tree, we can find such a set of representatives by choosing a lift of $\T/H$ to $\T$.  The windmill data---more specifically the $\calO_i$---exactly describe such a set of representatives, as shown in the proof of Proposition~\ref{prop:rephrase}.

The proof that we give of Theorem~\ref{thm:treecase} can be thought of as a rephrasing of the standard argument from Bass--Serre theory.  This rephrasing models our proof of the more general Theorem~\ref{thm:spin}.  Before giving the proof, we require the following technical lemma, which is also required for the proof of Theorem~\ref{thm:spin}.  

\ip{A consequence of distinctness of waypoints} In the proof of Theorems~\ref{thm:treecase} and~\ref{thm:spin} we will show inductively that the set of pivot points $\W(h)$ for an element $h \in F^{(i)}$ are distinct waypoints.  The inductive argument will use the following technical lemma.

\begin{lemma}
\label{lem:rephrase2}
Let $\P$ be a projection complex, $G$ a group that acts on $\P$, and $\{R_v\}$ an equivariant family of subgroups.  Choose windmill data and let $\{(F_i,\rho_i)\}_{i=0}^\infty$ be the associated free product data.  Fix $i \geq 0$, and suppose the pivot points $\W(h)$ are distinct waypoints for each $h \in F_i$.  If $h \in F_i$ and $\grp{h}$ has a bounded orbit, then $h$ lies in $R_w$ for some $w \in W_i$.
\end{lemma}

\begin{proof}
Consider $h \in F_i$ at some level $j \leq i$.  If $\grp{h}$ has a bounded orbit, then the number of waypoints in $\W(h^n)$ is bounded, independently of $n$, and therefore the length of the syllable decomposition of $h^n \in F^{(j)}$ is also bounded, independently of $n$.  This means that $h$ is conjugate to an element with syllable length at most 1.  Since $h \in F^{(j)}$, that means $h$ lies in $gR_vg^{-1}$ with $v \in \cup_{j=0}^{i} \calO_j$ where $j \leq i$ and $g \in F_i$.  Therefore $h$ lies in $R_w$ for $w = gv$.  
\end{proof}

\begin{proof}[Proof of Theorem~\ref{thm:treecase}]

As in the definition of the windmill data, $W_0$ consists of a basepoint, which we denote $v_0$.  Also for $i \geq 1$, we denote the 1--neighborhood of $W_{i-1}$ in $\T$ by $N_i$ and $L_i = N_i \setminus W_{i-1}$.  Let $\{(F_i,\rho_i)\}_{i=0}^\infty$ be the free product data associated to the given windmill data, let $F$ denote the direct limit of the $F_i$, let $H$ denote the direct limit of the $H_i$, and let $F^{(i)}$ denote the $i$th level of $F$.  

As per Proposition~\ref{prop:pivotway}, it is enough to show for each $i \geq 1$ and each $h \in F^{(i)}$ that some element of $\W(h)$ is a waypoint.  To prove this for all $i \geq 1$, we prove the following stronger statement by induction on $i \geq 0$: for each element $h \in F_i$, the pivot points of $\W(h)$ are waypoints and they appear in order along the geodesic from $v_0$ to $hv_0$ (in particular, they are distinct, which will allow us to apply Lemma~\ref{lem:rephrase2}).  To mirror the proof of Theorem~\ref{thm:spin} in Section~\ref{sec:proof}, we break up the inductive hypothesis into two statements:
\begin{enumerate}
\item[(A)] If $h \in F_i$ and $w \in \W(h)$, then $w$ is a waypoint for $h$.
\item[(B)] If $h \in F_i$, the elements of $\W(h)$ appear in order along the geodesic from $v_0$ to $hv_0$.
\end{enumerate}

The base cases of (A) and (B), when $i=0$, are vacuous.  So henceforth assume that $i \geq 1$, and that our inductive hypothesis holds for all $h$ with level less than $i$.  We now give the inductive steps for (A) and (B) in turn.

\medskip \noindent \emph{Proof of (A).} It suffices to consider $h \in F^{(i)}$.  We consider a secondary induction on the syllable length $n$ of $h$.  So suppose that $h$ has syllable decomposition $h_1 \cdots h_n$ where $n \geq 0$.  For the secondary induction, our base case is $n=0$. Since $i \geq 1$ and the identity lies in $F^{(-1)}$, there is no element of $F^{(i)}$ with syllable length 0, and so the base case for the secondary induction is vacuous.

Assume now that $n \geq 1$.  Let $w_k \in \W(h)$.  The pivot point $w_k$ corresponds to some syllable $h_k$ that lies at level $i$.  Let $h_\sigma = h_1 \cdots h_{k-1}$ and $h_\tau = h_{k+1} \cdots h_n$ ($\sigma$ and $\tau$ stand for ``starting'' and ``terminal'').  Let $v_k$ denote the vertex of $\calO_i$ with $h_k \in R_{v_k}$.  With this notation in hand, we write
\[
h = h_\sigma h_k h_\tau
\qquad 
\text{ and so }
\qquad 
w_k = h_\sigma  v_k.
\]
Applying $h_\sigma^{-1}$ to all vertices, statement (A) is equivalent to the statement that $v_k$ lies strictly between $h_\sigma^{-1} v_0$ and $h_kh_\tau v_0$.  

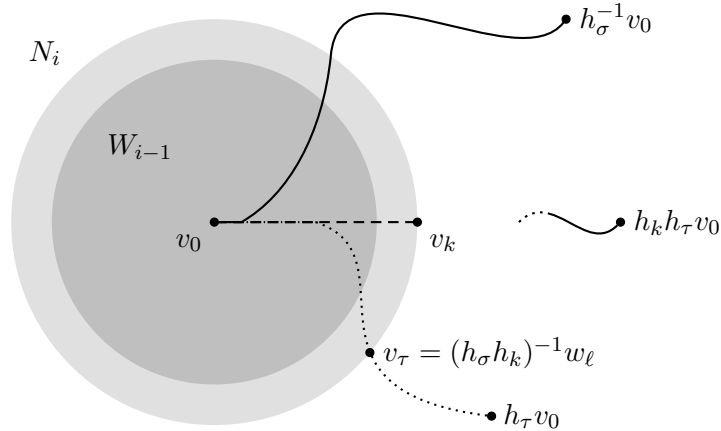
\begin{figure}
\begin{center}
\begin{tikzpicture}[scale=0.9]
\def\s{0.06}
\fill[black!12!white] (0,0) circle [radius=3cm];
\fill[black!25!white] (0,0) circle [radius=2.4cm];
\node at (-1.1,1.1) {$W_{i-1}$};
\node at (-2.5,2.5) {$N_{i}$};
\filldraw (0,0) circle (\s) node[inner sep=0,label=below left:{$v_0$}](v0) {};
\filldraw (3,0) circle (\s) node[inner sep=0,label=below right:{$v_k$}](vk) {};
\filldraw[rotate=-40] (3,0) circle (\s) node[inner sep=1,label=right:{$v_\tau = (h_\sigma h_k)^{-1}w_\ell$}](vt) {};
\filldraw (6,0) circle (\s) node[inner sep=1,label=right:{$h_k h_\tau v_0$}](hkht) {};
\filldraw[rotate=30] (6,0) circle (\s) node[inner sep=1,label=right:{$h^{-1}_\sigma v_0$}](hs) {};;
\filldraw[rotate=-35] (5,0) circle (\s) node[inner sep=0,label=right:{$h_\tau v_0$}](htv0) {};
\draw[thick,dash pattern=on 4 off 2.5 pt] (v0) -- (vk);
\draw[thick,dash pattern=on 4 pt off 1 pt on 0.5 pt off 1 pt] (v0) -- (1.5,0);
\draw[thick] (v0) -- (0.4,0) .. controls (1.1,0.4) and (1.6,1.3) .. (1.721,2.457) .. controls  (1.85,4) and (4.5,2.1) .. (hs);
\draw[thick,dotted] (1.5,0) .. controls  (3,-0.5) and (1,-2.5) .. (htv0);
\draw[thick,dash pattern=on 1 pt off 2 pt on 1 pt off 2pt on 1pt off 2pt on 1 pt off 2pt on 500pt] (4.5,0) .. controls (5,0.5) and (5.5,-0.5) .. (hkht);
\end{tikzpicture}
\end{center}
\caption{The vertices and paths used in the proof of Theorem~\ref{thm:treecase}.}
\label{fig:tree}
\end{figure}

Consider the diagram in Figure~\ref{fig:tree}, which shows the case when $h_\sigma, h_\tau \in F^{(i)}$.  There we have drawn the geodesics from $v_0$ to the vertices $h_\sigma^{-1} v_0$, $v_k$, and $h_\tau v_0$, respectively.  These paths are solid, dashed, and dotted, respectively.  The paths might overlap near $v_0$, as suggested by the picture.  Let $P$ be the geodesic from $h_\sigma^{-1} v_0$ to $v_k$ and let $Q$ be the geodesic from $v_k$ to $h_\tau v_0$.  Since $h_k$ fixes $v_k$, we have that $h_kQ$ is the geodesic path from $v_k$ to $h_kh_\tau v_0$ and moreover that the concatenation of $P$ and $h_kQ$ is a path from $h_\sigma^{-1} v_0$ to $h_kh_\tau v_0$ that contains $v_k$.  

To prove statement (A), then, it is enough to prove that both $P$ and $Q$ are nontrivial and that the concatenation of $P$ and $h_kQ$ is a geodesic.  To do this, it is enough to show that both $P$ and $Q$ contain the (unique) edge $e$ connecting $v_k$ to $W_{i-1}$; equivalently that $P$ and $Q$ each contain a vertex of $N_i$ other than $v_k$.  It then follows from the spinning hypothesis that $h_k(e) \neq e$, and hence that the concatenation of $P$ and $h_kQ$ is geodesic near $v_k$, hence is a geodesic.  We treat the case of $Q$, with the case of $P$ being the essentially the same.  

There are two cases, according to whether $h_\tau$ lies in $F_{i-1}$ or in $F^{(i)}$.  In the first case, we have that $h_\tau v_0$ lies in $W_{i-1}$.  Since $h_\tau v_0$ is a vertex of $Q$, this completes the proof in the first case.  

The second case is where $h_\tau$ lies in $F^{(i)}$.  Let $v_\tau$ be the first pivot point for $h_\tau$.  We claim that $v_\tau$ lies on $Q$ (as shown in Figure~\ref{fig:tree}).  Indeed, by the induction on syllable length applied to $h_\tau$, we have that $v_\tau$ lies on the geodesic from $v_0$ to $h_\tau v_0$.  Further, since $v_\tau$ is a leaf of the tree $N_i$, and since $v_0$ lies in $N_i$, any geodesic from a vertex of $N_i$ to $h_\tau v_0$ contains $v_\tau$, in particular the path $Q$ contains $v_\tau$ (recall that $Q$ contains $v_k \in N_i$).  Thus, to show that $Q$ contains a vertex of $N_i$ other than $v_k$, it suffices to show the following claim.

\medskip \noindent \emph{Claim.}  The vertex $v_\tau$ is distinct from $v_k$.\medskip

\noindent From the definitions, $v_\tau$ is the fixed point of the first syllable of $h_\tau$ that is not contained in $F_{i-1}$.  If this syllable is $h_{k+1}$, then by the definition of syllable decomposition we have $v_k \neq v_{k+1} = v_\tau$.  

The other case of the claim is where $h_{k+1} \in F_{i-1}$.  In this case we must show that $v_k \neq h_{k+1}v_{k+2} = v_\tau$.  Assume for the sake of contradiction that $v_k = h_{k+1}v_{k+2}$.  To obtain the contradiction, we will show that $h_{k+1}$ fixes $v_k$, that is also lies in some $R_{w}$, and further that $w \neq v_k$.  Together, these three items violate the spinning hypothesis.

Since $v_k$ and $v_{k+2}$ are vertices of $\calO_i$, which is a set of orbit representatives for the action of $H_{i-1}$ on $L_i$, it follows from the assumption $v_k = h_{k+1}v_{k+2}$ that $v_k = v_{k+2}$.  Thus $v_k$ is fixed by $h_{k+1}$, which is the first item.  By the inductive hypothesis that all pivot points are waypoints for elements in $F_{i-1}$ and since $h_{k+1}$ fixes $v_k$, it follows from Lemma~\ref{lem:rephrase2} that $h_{k+1}$ lies in some $R_{w}$ with $w \in W_{i-1}$, which is the second item.  Since $v_k$ lies in $L_i$,  and since the latter is disjoint from $W_{i-1}$, we obtain the third item. This completes the proof of the claim, hence (A).

\medskip 

\medskip \noindent \emph{Proof of (B).}   Take $h \in F^{(i)}$ as in the proof of (A).  It suffices to show that if $w_k$ is an element of $\W(h)$ that is not last in the order, and if $w_\ell$ is the immediately following element of $\W(h)$, then $w_\ell$ lies strictly between $w_k$ and $hv_0$ in $\T$.    As above, we fix the decomposition $h_\sigma h_k h_\tau$ of $h$ and we take $v_k$ to be the vertex of $\T$ associated to $h_k$.  We have that $w_\ell = h_\sigma h_k v_\tau$ where $v_\tau$ is the first pivot point for $h_\tau$.   We may assume that $h_\tau$ lies in $F^{(i)}$, for otherwise $w_k$ would be the final pivot point.  Similarly to our restatement of (A), we may apply $(h_\sigma h_k)^{-1}$ to all vertices in statement (B) in order to obtain the equivalent statement that $v_\tau$ lies strictly between $v_k$ and $h_\tau v_0$.  This is further equivalent to the statement that $v_\tau$ lies in the interior of $Q$, which is the geodesic from $v_k$ and $h_\tau v_0$.  We observed this fact above, and so the proof is complete.
\end{proof}

\section{Proof of main technical theorem}\label{sec:proof}

In this section we prove our main technical theorem, Theorem \ref{thm:spin}.  This theorem states that if a group $G$ acts on a projection complex $\P$ with an equivariant $L$--spinning family of subgroups $\{R_v\}_{v\in\P}$ and $L$ is sufficiently large, then the normal subgroup generated by $\{R_v\}_{v\in\P}$ is a free product of some of the $R_v$.

Our proof comes after a lemma that speaks to the difference between Theorem \ref{thm:spin} and the special case considered in Theorem~\ref{thm:treecase} about trees.  The proofs of both theorems use windmills in virtually the same way, but the general case is complicated by the fact that windmills in a general projection complex need not be convex.  Rather, we have Lemma~\ref{Theta} below, which is critical to our proof.  

\p{A convexity-like property of windmills in projection complexes} In the statement of the following lemma, $\Kp$ is the path constant associated to $\P$, as in Subsection~\ref{subsec:def and prop}.  

\begin{lemma}\label{Theta}
Let $\P$ be a projection complex, $G$ a group that acts on $\P$, and $\{R_v\}$ an equivariant family of subgroups with windmill data $\{(H_i,W_i,\calO_i)\}_{i=0}^{\infty}$.  If $x$ and $y$ are vertices of $W_{i-2}$ and $v \notin W_i$, then $d_v(x,y) \leq \Kp$.
\end{lemma}

\begin{proof}

For $i \geq 2$, the distance between any vertex of $\P \setminus W_i$ and any vertex of $W_{i-2}$ in $\P$ is at least 3.  This follows from the definitions, since each $W_i$ contains the 1-neighborhood of $W_{i-1}$.  Because $W_{i-2}$ is connected, the lemma now follows from the bounded path image property.
\end{proof}

\begin{proof}[Proof of Theorem~\ref{thm:spin}]

As in the statement, $G$ is a group acting on a projection complex $\P$.  Let $\theta$ be the constant from the definition of $\P$, and let $\Ke$, $\Kg$, and $\Kp$ be the associated constants from Section~\ref{subsec:def and prop}, all of which ultimately depend on $\theta$.  We will prove the theorem for $L(\P) = 3(11\Ke + 6\Kg + 5\Kp + \theta) + 1$.  In other words, we assume that $G$ has an equivariant $L$--spinning family of subgroups $\{R_v\}$ with $L \geq L(\P)$, and show that the normal closure of the union of the $R_v$ can be decomposed into the free product described in the statement.  As in Section~\ref{sec:windmillsprojection}, let $\{H_i, W_i,\calO_i\}$ be windmill data associated to the action of $G$ on $\P$ and let $\{F_i,\rho_i\}$ be the associated free product data.  As in the definition of the windmill data, $W_0$ consists of a basepoint, which we denote $v_0$.  Also for $i \geq 1$, we denote the 1--neighborhood of $W_{i-1}$ in $\P$ by $N_i$ and $L_i = N_i \setminus W_{i-1}$. 

Let $m = 11\Ke + 6\Kg + 5\Kp$; in particular $L(\P) = 3(m+\theta) + 1$.  To prove the theorem we will show by induction on $i$ that the following statements hold for all $i \geq 0$.

\begin{itemize}
\item[(A)] If $h \in F_i$ and $w \in \W(h)$, then $d_w(v_0,hv_0) > m+\theta$.
\item[(B)] If $h \in F_i$, the elements of $\W(h)$ appear in order along any geodesic from $v_0$ to $hv_0$.
\item[(C)] For $x \in N_{i+1}$ and $v \notin W_{i}$ with $v\neq x$, we have $d_v(v_0,x) \leq m$.
\end{itemize}

Respectively, these statements say that pivot points are waypoints, waypoints are distinct, and vertices outside a given $W_i$ see bounded projection between points inside $W_i$ (even $N_{i+1}$), which strengthens the convexity-like property of Lemma~\ref{Theta}.  The precise relevance of each statement to the proof is as follows.  By the bounded geodesic image property, statement (A) implies that the pivot points of $h$ are waypoints, since $m +\theta \geq \Kg$.  By Proposition~\ref{prop:pivotway}, this will complete the proof of the theorem, once verified for all $i$.  Statement (B) further implies that the waypoints are distinct, which will allow us to apply Lemma~\ref{lem:rephrase2} in the same way as in the proof of Theorem~\ref{thm:treecase}.  Statement (C) enables induction, as in effect it upgrades to $W_i$ the $W_{i-2}$ in Lemma~\ref{Theta}.

For the base case of our induction on $i$, we consider the level $i=0$.  Statements (A) and (B) are vacuous.  The hypotheses of (C) imply that $v \notin \{x,v_0\}$.  Because $d(v_0,x)=1$, we have $d_v(v_0,x)\leq \Ke \leq m$, where $\Ke$ is the edge constant from Section \ref{subsec:def and prop}.

Now consider the level $i\geq 1$ and assume the three statements hold for smaller levels.  We prove each statement in turn.  First, we prove (A) by our assumption on lower levels, and inducting on the syllable length of $h$.  Then we upgrade (A) to (B).  These are key to proving the new level of (C), which is in turn used to prove the next level of (A) and (B).

\medskip \noindent \emph{Proof of (A).}  It suffices to consider $h \in F^{(i)}$.  We begin a second induction on the syllable length of $h$, using the decomposition $h=h_1\cdots h_n$ described in Section \ref{sec:pivot-way}.  For the base case of the induction on syllable length, suppose $n=0$.  Since $i \geq 1$ and the identity lies in $F^{(-1)}$, there is no element of $F^{(i)}$ with syllable length 0, and so the base case for the secondary induction is vacuous.    

Now suppose that $n \geq 1$ and that (A) holds for elements of $F^{(i)}$ with syllable length less than $n$.  Let $w_k \in \W(h)$ be a pivot point for $h$, so $w_k=h_1\cdots h_{k-1}v_k$ where $v_k \in \calO_i$ is fixed by $h_k$.   As in Section~\ref{sec:trees}, let $h=h_\sigma h_k h_\tau$, where $h_\sigma$ and $h_\tau$ are possibly trivial subwords.  
  
\medskip \noindent {\it Step 1: The quantity $d_{w_k}(v_0,hv_0)$ is defined, as are $d_{v_k}(v_0,h_\sigma^{-1}v_0)$ and $d_{v_k}(v_0,h_\tau v_0)$.}  We will use the latter two quantities in Step 2.  The quantity $d_{w_k}(v_0,hv_0)$ is defined if and only if $w_k \notin \{v_0,hv_0\}$.  We start by showing that $w_k \neq hv_0$.  There are two cases: either $h_\tau \in F_{i-1}$ or  $h_\tau \in F^{(i)}$.  In the first case, we have $h_\tau v_0 \in F_{i-1} \cdot N_{i-1} = W_{i-1}$.  As $v_k \in L_i$, which is disjoint from $W_{i-1}$, we have $v_k \neq h_\tau v_0$ and therefore $w_k = h_\sigma v_k = h_\sigma h_k v_k \neq hv_0$ as desired.
  
The second case is where $h_\tau$ lies in $F^{(i)}$.  Let $v_\tau$ be the first pivot point for $h_\tau$.  By induction on $n$, statement (A) gives $d_{v_\tau}(v_0, h_\tau v_0) > m + \theta$.  We observe that both $v_\tau$ and $v_k$ are elements of $L_i$.  We claim that $v_\tau \neq v_k$.  This is exactly the claim from the proof of Theorem~\ref{thm:treecase}, whose proof uses the inductive hypothesis that statements (A) and (B) are true for the lower levels to invoke Lemma~\ref{lem:rephrase2}.  In this context the same proof works after replacing ``spinning'' with ``$L$--spinning.''  Hence we may apply statement (C) inductively (with respect to $i$) to conclude $d_{v_\tau}(v_0,v_k) \leq m$.  Since we have shown that $d_{v_\tau}(v_0, h_\tau v_0) > m + \theta$ and that $d_{v_\tau}(v_0,v_k) \leq m$, it follows that $v_k \neq h_\tau v_0$ and therefore, as in the first case, we have $w_k \neq hv_0$.  

To confirm that $d_{w_k}(v_0,hv_0)$ is defined, it remains to show that $w_k \neq v_0$.  To see this, we first observe that $h^{-1}w_k$ is a pivot point for $h^{-1}$.  Hence the above argument shows that $h^{-1}w_k \neq h^{-1}v_0$ and thus $w_k \neq v_0$ as well.    

To address the quantities $d_{v_k}(v_0,h_\sigma^{-1}v_0)$ and $d_{v_k}(v_0,h_\tau v_0)$, we observe that $v_k$ is a pivot point for $h_k h_\tau$, which lies in $F^{(i)}$ and has syllable length at most $n$.  Thus, applying the same argument above to $h_kh_\tau$ instead of $h$, we deduce that $v_k \notin \{v_0,h_k h_\tau v_0\}$.  In particular, $v_k \neq v_0$.  Further, since $v_k$ is fixed by $h_k$ we also have $v_k \neq h_\tau v_0$.  Similarly we have $v_k \notin \{v_0,h_\sigma^{-1}v_0\}$ as well. 

\medskip \noindent {\it Step 2: We have $d_{w_k}(v_0,h v_0) > m + \theta$.} Using the invariance of the distance functions under the group action, the triangle inequality, and the $L$--spinning hypothesis, we have 
\begin{equation*}\label{eqw}
\begin{split}
d_{w_k}(v_0,hv_0) &= d_{v_k}(h_\sigma^{-1} v_0,h_kh_\tau v_0) \\
&\geq d_{v_k}(v_0,h_kv_0) - d_{v_k}(v_0,h_\sigma^{-1} v_0) - d_{v_k}(h_kv_0,h_kh_\tau v_0) \\
&\geq L - (d_{v_k}(v_0,h_\sigma^{-1} v_0) + d_{v_k}(v_0,h_\tau v_0)).
\end{split}
\end{equation*}

We will show that $d_{v_k}(v_0,h_\tau v_0) \leq m + \theta$.  Since $h_\sigma^{-1}$ has functionally identical features (replacing $h$ with $h^{-1}$), the same argument also shows that $d_{v_k}(v_0,h_\sigma^{-1} v_0) \leq m + \theta$.  Incorporating this into above inequality, we will conclude that
\begin{equation}
d_{w_k}(v_0,hv_0) \geq d_{v_k}(v_0,h_kv_0) - 2(m + \theta) \geq L - 2(m+\theta) > m+\theta, \label{eq:callout}
\end{equation}
which will complete the proof of (A).

Towards bounding $d_{v_k}(v_0,h_\tau v_0)$, there are the same two cases as in Step 1: either $h_\tau \in F_{i-1}$ or $h_\tau \in F^{(i)}$.  In the first case, we observed in Step 1 that $h_\tau v_0 \in W_{i-1} \subseteq N_i$, $v_k \notin W_{i-1}$ and that $v_k \neq h_\tau v_0$.  Thus by induction on level $i$, statement (C) implies that $d_{v_k}(v_0,h_\tau v_0)\leq m \leq m + \theta$.  

It remains to consider the second case.  Let $v_\tau$ be the first pivot point in $\W(h_\tau )$.  In this case, we have already observed in Step 1 that $d_{v_\tau}(v_0,h_\tau v_0) > m + \theta$ and $d_{v_\tau}(v_0,v_k)\leq m$.  We thus have
\[d_{v_\tau}(v_k,h_\tau v_0) \geq d_{v_\tau}(v_0,h_\tau v_0) - d_{v_\tau}(v_0,v_k) > (m+\theta) - m > \theta.  \]
Using the inequality on triples, we obtain $d_{v_k}(v_\tau,h_\tau v_0) \leq \theta$.  As shown in Step 1, $v_k$ and $v_\tau$ are distinct vertices of $L_i$, we apply statement (C) inductively (with respect to $i$) to obtain that $d_{v_k}(v_0,v_\tau)\leq m$.  We thus have
\[d_{v_k}(v_0,h_\tau v_0) \leq d_{v_k}(v_\tau,h_\tau v_0) + d_{v_k}(v_0,v_\tau) \leq \theta + m,\]
completing the proof of (A).

\medskip \noindent \emph{Proof of (B).}  Recall that $\W(h)=\{w_j\}_{j \in I}$ is an ordered set with index set $I$ a (possibly proper) subset of $\{1,\ldots,n\}$. 
Our goal is to show that these vertices appear in order on any geodesic from $v_0$ to $hv_0$.  

Let $w_k,w_\ell \in\W(h)$ be consecutive, so $w_k = h_\sigma v_k$ and $w_\ell = h_\sigma h_k  v_\tau$ where $v_\tau$ is the first pivot point of $h_\tau$, again using the syllable decomposition $h = h_\sigma h_k h_\tau$.  As above, we have that $v_k \neq v_\tau$ and thus $w_k \neq w_\ell$.  Using the triangle inequality, statement (A) and invariance of the distance functions, we have
\[d_{w_k}(v_0,w_\ell) \geq d_{w_k}(v_0,hv_0) - d_{w_k}(w_\ell,hv_0) \geq m+\theta - d_{v_k}(v_\tau,h_\tau v_0). \]
As in the proof of (A), using the inequality on triples, we may deduce that $d_{v_k}(v_\tau,h_\tau v_0)\leq \theta$, and therefore $d_{w_k}(v_0,w_\ell)\geq m > \Kg$.  By the bounded geodesic image property, any geodesic from $v_0$ to $w_\ell$ passes through $w_k$.  

\medskip \noindent \emph{Proof of (C).}  Recall we are given $v \notin W_i$ and $x\in N_{i+1}$ such that $x \neq v$, and our goal is to bound $d_v(x,v_0)$ by $m = 11\Ke + 6\Kg + 5\Kp$.  There is an $x_0 \in W_i$ with $d(x,x_0) \leq 1$.  If there is a geodesic from $x_0$ to $v_0$ that avoids $v$---in particular if $W_i$ is convex---then by the bounded edge and geodesic image properties we have
\[
d_v(v_0,x) \leq d_v(v_0,x_0) + d_v(x_0,x) \leq \Kg + \Ke \leq m.
\]
However, it might be the case that every geodesic from $x_0$ to $v_0$ passes through $v$.  In this case, we build a path from $x_0$ to $v_0$ consisting of: at most $10$ edges avoiding $v$, at most $6$ geodesics avoiding $v$, and at most $5$ paths that each lie in some $F_i$--translate of $W_{i-2}$.  Since $v \notin W_i$ and $W_i$ is $F_i$--invariant, the projection of any path in an $F_i$--translate of $W_{i-2}$ to $v$ is bounded by $\Kp$ by Lemma~\ref{Theta}.  Using this fact and the bounded edge and geodesic image properties, it follows that
\[
d_v(v_0,x) \leq d_v(v_0,x_0) + d_v(x_0,x) \leq (10\Ke + 6\Kg + 5\Kp) + \Ke = m.
\]
Figure~\ref{fig:proofC} illustrates the idea.  In the rest of the proof we refer to edges avoiding $v$, geodesics avoiding $v$, and paths lying in an $F_i$--translate of $W_{i-2}$ as paths of type e, g, and p, respectively.  Also, we assume throughout that $i > 1$; the case $i=1$ follows the same idea but is simpler.

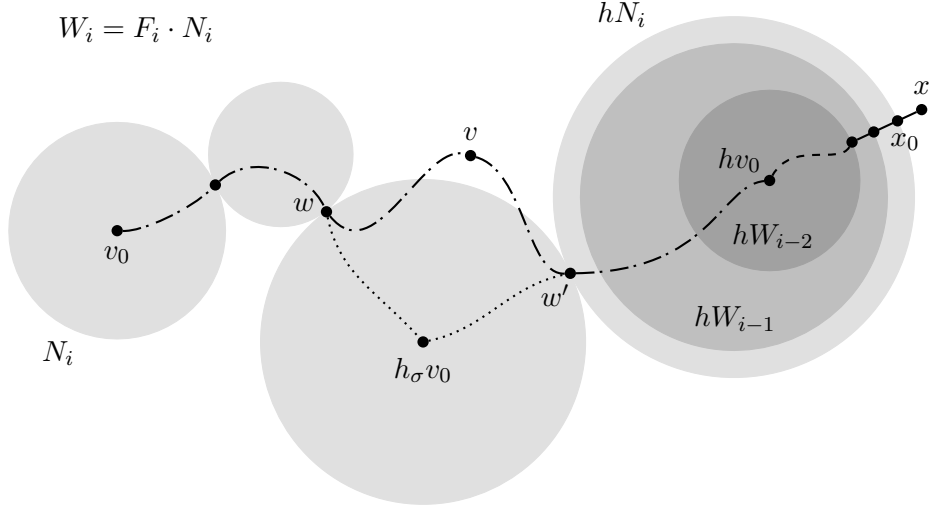
\begin{figure}
\begin{center}
\begin{tikzpicture}
\begin{scope}[rotate=-20]
\def\r{2.4}
\tikzset{vertex/.style={fill,black,circle,minimum size=0.15cm,inner sep=0pt,outer sep =0pt}}
\filldraw[black!12!white] (${sqrt(2)}*(0.2*\r,0)$) circle [radius=0.6*\r cm];
\filldraw[black!12!white] (${sqrt(2)}*(0.7*\r,0.5*\r)$) circle [radius=0.4*\r cm];
\filldraw[black!12!white] (${sqrt(2)}*(1.475*\r,0)$) circle [radius=0.9*\r cm];
\filldraw[black!12!white] (${sqrt(2)}*(2.425*\r,0.95*\r)$) circle [radius=\r cm];
\filldraw[black!25!white] (${sqrt(2)}*(2.425*\r,0.95*\r)$) circle [radius=0.85*\r cm];
\filldraw[black!37!white] (${sqrt(2)}*(2.425*\r,0.95*\r)+(0.37,0.37)$) circle [radius=0.5*\r cm];
\node[vertex,label=below:{$v_0$}](v) at (${sqrt(2)}*(0.2*\r,0)$) {};
\node[vertex,label={[shift={(-0.3,-0.2)}]$w$}](w) at (${sqrt(2)}*(0.945*\r,0.35*\r)$) {};
\node[vertex,label={[shift={(-0.2,-0.65)}]$w'$}](w') at (${sqrt(2)}*(1.925*\r,0.45*\r)$) {};
\node[vertex,label=below:{$h_\sigma v_0$}](hsv) at (${sqrt(2)}*(1.475*\r,0)$) {};
\node[vertex,label={[shift={(-0.4,-0.1)}]$hv_0$}](hv) at (${sqrt(2)}*(2.425*\r,0.95*\r)+(0.37,0.37)$) {};
\node[vertex,label=above:{$x$}](x) at (${sqrt(2)}*(2.925*\r,1.45*\r)+(0.25,0.25)$) {};
\node[vertex,label={[shift={(0.1,-0.6)}]$x_0$}](x0) at (${sqrt(2)}*(2.925*\r,1.45*\r)$) {};
\node[vertex](x1) at (${sqrt(2)}*(2.425*\r,0.95*\r)+(1.45,1.45)$) {};
\node[vertex](x2) at (${sqrt(2)}*(2.425*\r,0.95*\r)+(1.225,1.225)$) {};
\draw[thick] (x) -- (x2);
\draw[thick,dashed] (x2) .. controls (9.4,4.1) and (8.7,4.2) .. (hv);
\node[vertex] (w0) at (${sqrt(2)}*(0.5*\r,0.3*\r)$) {};
\node[vertex,label=above:{$v$}](vv) at (${sqrt(2)}*(1.4*\r,0.75*\r)$) {};
\draw[thick,dash pattern={on 7pt off 2pt on 1pt off 3pt}] (v) .. controls (1,0.1) and (1.5,0.6) .. (w0);
\draw[thick,dash pattern={on 7pt off 2pt on 1pt off 3pt}] (w0) .. controls (2,1.6) and (2.8,1.6) .. (w);
\draw[thick,dash pattern={on 7pt off 2pt on 1pt off 3pt}] (w) .. controls (4.1,0.65) and (4.1,2.5) .. (vv);
\draw[thick,dash pattern={on 7pt off 2pt on 1pt off 3pt}] (vv) .. controls (5.5,2.6) and (6,1.3) .. (w');
\draw[thick,dash pattern={on 7pt off 2pt on 1pt off 3pt}] (w') .. controls (8.3,2.2) and (8,3.3) .. (hv);
\draw[thick,dotted] (w) to[out=-60,in=150] (hsv);
\draw[thick,dotted] (hsv) to[out=30,in=-150] (w');
\node at (6,5) {$h N_i$};
\node at (0,2.6) {$W_i = F_i \cdot N_i$};
\node at (8.8,1.7) {$h W_{i-1}$};
\node at (8.9,2.9) {$h W_{i-2}$};
\node at (0.5,-1.8) {$N_i$};
\end{scope}
\end{tikzpicture}
\caption{The picture for steps 1 and 2 of statement (C) in the proof of Theorem~\ref{thm:spin}.}
\label{fig:proofC}
\end{center}
\end{figure}

\medskip \noindent {\it Step 1: Connect $x_0$ to $hv_0$ for some $h \in F_i$.}  Using the definition of a windmill, there is an $h \in F_i$, $x' \in W_{i-2}$ and a path of length 2 from $x_0$ to $hx'$ that does not include $v$ (two paths of type e).  As $W_{i-2}$ is connected, there is a path in $hW_{i-2}$ from $hv_0$ to $hx'$ (path of type p). 

\medskip \noindent {\it Step 2: Replace a segment of a geodesic from $v_0$ to $hv_0$ if it passes through $v$.}  If there is a geodesic from $v_0$ to $hv_0$ avoiding $v$ we are done, since this is a path of type g.  Otherwise, $v$ lies on every geodesic from $v_0$ to $hv_0$.  Fix such a geodesic.  There are consecutive points $w,w' \in \W(h) \cup \{v_0,hv_0\}$ such that $v$ lies between $w$ and $w'$ on this geodesic (extend the order on $\W(h)$ so $v_0$ is minimal and $hv_0$ is maximal).  The (possibly trivial) sub-geodesics from $v_0$ to $w$ and from $w'$ to $hv_0$ are of type g.  By definition of the pivot points, $w$ and $w'$ are in $h_\sigma N_i$ for an initial subword $h_\sigma$ of $h$.  We replace the segment of the geodesic from $w$ to $w'$ by a concatenation of geodesics from $w$ to $h_\sigma v_0$ and from $h_\sigma v_0$ to $w'$.  If both of these geodesics avoid $v$ (so they are of type g), then we are done.  Suppose, on the contrary, that one or both passes through $v$.  We treat the case that the geodesic from $w$ to $h_\sigma v_0$ passes through $v$; the other case is identical.

\medskip \noindent {\it Step 3: Connect $w$ to $h_\sigma h'v_0$ for some $h' \in F_{i-1}$.}  As $w \in h_\sigma N_i$, this is  identical to Step 1 with $h$, $x_0$, and $hv_0$ replaced by $h_\sigma$, $w$, and $h_\sigma h'v_0$.  Like Step 1, this step contributes two paths of type e and at most one path of type p.

\medskip \noindent {\it Step 4: Replace a segment of a geodesic from $h_\sigma v_0$ to $h_\sigma h'v_0$ if it is passes through $v$.}  We proceed as in Step 2.  The geodesic splits into two paths of type g and geodesic from $u$ and $u'$ that lie in $h_\sigma h'_\sigma N_{i-1}$ for some initial subword $h'_\sigma$ (here $u$ and $u'$ are analogous to $w$ and $w'$ from Step 2).  There is a path from $u$ to $u'$ consisting of 2 paths of type e and a path of type p.
\end{proof}

\section{Proof of the pseudo-Anosov case}\label{sec:niwpd-proof}

In this section we prove Theorem~\ref{thm:base}, which states that if a group $G$ acts on a hyperbolic space $X$ and 
$\{f_1,\ldots,f_m\} \subseteq G$ is a normally independent collection of NEC WPD elements, then for any $t \geq 0$ there is an $N$ so that for $n \geq N$ the group $\normalclosure{f_1^n,\dots,f_m^n}_{G}$ has the following properties:
\begin{enumerate}
\item it is isomorphic to $F_\infty$ and has a free basis consisting of conjugates of the $f_i^n$, and
\item the $X$--translation length of each nontrivial element is at least $t$.
\end{enumerate}
The main application in this paper is the case where the $f_i$ are pseudo-Anosov elements of the mapping class group and $X$ is the curve complex.  

The first statement of Theorem~\ref{thm:base} will be proved by applying Theorem~\ref{thm:spin} to the projection complex produced by Proposition~\ref{prop:many}.  To prove the second statement, we will need to understand translation lengths in terms of projections.  Before giving the proof of the theorem, we introduce two tools.

\ip{A distance formula for WPD elements} The following proposition gives a lower bound on distance in a hyperbolic space, in terms of projections to quasi-axis bundles.   In the statement, $d_{\alpha}$ is the distance function defined in Section~\ref{sec:eg}.   Also, for a constant $M \geq 0$ and $x \in \R$ we define $\cut{x}_{M}$ to be $x$ if $x \geq M$ and $0$ otherwise.    

\begin{proposition}\label{prop:mm-hyperbolic}
Let $X$ be a hyperbolic metric space, let $G$ be a group acting on $X$, and let $\{f_1,\dots,f_m\} \subseteq G$ be a collection of WPD elements.  Let $\A$ be the collection of quasi-axis bundles in $X$ for the $G$--conjugates of the $f_i$.  Then there exists a constant $M$ such that for any $x,y \in X$, we have
\begin{equation*}
d(x,y) \geq \frac{1}{6} \sum _{\alpha \in \A} \cut{d_{\alpha}(x,y) }_{M}.
\end{equation*}
\end{proposition}
The formula is reminiscent of the Masur--Minsky distance formula for elements in the mapping class group in terms of subsurface projections~\cite{ar:MM00}.  This inequality is surely well-known to experts (cf.~\cite[Theorem~4.13]{ar:BBF15}).  We provide a proof in a forthcoming paper \cite{ar:CMMexpos}.

\ip{A lower bound on projections to pivot points.} In order to apply Proposition~\ref{prop:mm-hyperbolic}, we need the following lemma, which says that the local spinning around a pivot point, $d_{v_k}(v_0,h_kv_0)$, is a lower bound on global projection distance, $d_w(v_0,hv_0)$.   In the statement, $L(\P)$ is the constant from Theorem~\ref{thm:spin}.  

\begin{lemma}\label{lem:remember}
Let $\P$ be a projection complex, $G$ a group that acts on $\P$, and $\{R_v\}$ an equivariant $L$--spinning family subgroups with $L \geq L(\P)$.  Choose windmill data and let $F$ be the associated free product.  Let $h \in F$ and let $h_1 \cdots h_\ell$ be the syllable decomposition of $h$.  Let $1 \leq k \leq \ell$ and let $w_k = h_1 \cdots h_{k-1}v_k \in \W(h)$.  Then
\[ d_{w_k}(v_0,hv_0) \geq \frac{1}{3}d_{v_k}(v_0,h_kv_0). \]
\end{lemma}

\begin{proof}
By the inequality \eqref{eq:callout} from the proof of Theorem~\ref{thm:spin}, we have that
\begin{equation*}
d_{w_k}(v_0,hv_0) \geq d_{v_k}(v_0,h_k v_0) - 2L/3.
\end{equation*}
As $d_{v_k}(v_0,h_k v_0) \geq L$, the proposition follows.
\end{proof}

\begin{proof}[Proof of Theorem \ref{thm:base}]
We begin with the first statement.  For each $f_i$ we fix a quasi-axis bundle $\beta_i \subseteq X$.  Let $\Y$ be the set of $G$--translates of these subsets of $X$.  By Proposition~\ref{prop:many}, the set $\Y$ and the distance functions $\{d_\beta\}_{\beta \in \Y}$ satisfy the projection complex axioms; let $\P$ be this projection complex.  Let $L = L(\P)$ be the constant from Theorem~\ref{thm:spin} and let $\tau$ be the minimum translation length of the $f_i$.  Then, as long as $N \geq L/\tau$,  the first statement is an application of Theorem~\ref{thm:spin}.  It remains to prove the second statement.

Towards proving the second statement, let $\Delta \geq 0$ be such that $\diam \pi_{\beta}(\beta') \leq \Delta$ for all distinct $\beta,\beta' \in \Y$.  (The existence of $\Delta$ is implicit in Proposition~\ref{prop:many}; a direct reference for the existence also appears in the paper by Dahmani--Guirardel--Osin \cite[Lemma~4.46]{ar:DGO17}.)   Also, let $M$ be the constant from Proposition~\ref{prop:mm-hyperbolic} applied to $X$ and the collection of quasi-axis bundles $\Y$.  We will show that the second statement, hence the theorem, holds for $N = \frac{1}{\tau}\max\{L,24\Delta,4M,24t\}$.  To this end, we let $n \geq N$.

Choose windmill data $\{H_i,W_i,\calO_i\}$ and say $W_0 = \{v_0\}$.  Suppose $h \in \normalclosure{f_1^n,\ldots,f_m^n}_G$ is nontrivial.  Let $h_{1} \cdots h_{\ell}$ be the syllable decomposition of $h$ with respect to the free product data associated to the given windmill data.  Since translation length is a conjugacy invariant, we may assume that $h$ is cyclically reduced.  We treat in turn the cases where the syllable length of $h$ is 1 and where it is greater than 1.

If $h$ has syllable length 1, then $h$ is a power of some $f_i^n$, and so for any $x \in X$ we have $d(x,h^{p}x) \geq pn\tau$ for all $p \geq 0$.  Thus as $n\tau \geq t$ the translation length of $h$ is at least $t$, as desired.

Now assume that the syllable length of $h$ is at least 2.  Let $x \in X$ be a point in the quasi-axis bundle $v_0 \in \Y$.  To bound $d(x,hx)$ from below, we will apply Proposition~\ref{prop:mm-hyperbolic}, focusing only on the summands corresponding to the quasi-axis bundles in $X$ that are pivot points for $h$.  In order to estimate these terms from below, we will apply Lemma~\ref{lem:remember}.

Let $w_k = h_1\cdots h_{k-1}v_k$ be an element of $\W(h)$.  As $h_k$ is a power of some $f_i^n$, it follows from Lemma~\ref{lem:remember} that
\begin{align*}
d_{w_k}(v_0,hv_0) \geq \frac{1}{3}d_{v_k}(v_0,h_kv_0) \geq \frac{n\tau}{3}.
\end{align*}
Hence, as the diameters of $\pi_{w_k}(v_0)$ and $\pi_{w_k}(hv_0)$ are at most $\Delta$, we have \[d_{w_k}(x,hx) \geq d_{w_k}(v_0,hv_0) - 2\Delta \geq \frac{n\tau}{3} - \frac{n\tau}{12} = \frac{n\tau}{4}. \]
Hence by Proposition~\ref{prop:mm-hyperbolic}, the fact that $\frac{n\tau}{4} \geq M$, and the fact that $\frac{n\tau}{24} \geq t$, we find
\begin{equation*}
d(x,hx) \geq \ \frac{1}{6}\sum_{w \in \W(h)} \cut{d_{w}(x,hx)}_{M} \ \geq \ \frac{1}{6}\sum_{w \in \W(h)} \frac{n\tau}{4} \ \geq \abs{\W(h)}t.
\end{equation*}
As $h$ is cyclically reduced and has at least two syllables, we have that $\abs{\W(h^{p})} = p\abs{\W(h)}$ for $p \geq 0$ and hence the above argument applied to $h^{p}$ shows that $d(x,h^{p}x) \geq p\abs{\W(h)}t$.  Thus the translation length of $h$ in $X$ is at least $t$, as desired.
\end{proof}


\section{Proof of the main theorem}\label{sec:general-proof}

In this section we prove Theorem~\ref{thm:main}, which states that if $G$ is a subgroup of $\Mod(S)$, $\X$ is a $G$--overlapping family of subsurfaces, $\F$ is a finite, $G$--independent family of mapping classes that are carried by $\X$ and are NEC in $G$, then there is an $N > 0$ with the following properties:
\begin{enumerate}
\item for each $n \geq N$ and any set of orbit representatives $\calY$ for the action of $\normalclosure{\F^{(n)}}_{G}$ on $G \cdot \X$ we have 
\begin{equation*}
\normalclosure{\F^{(n)}}_{G} \cong \BigFreeProd{Y \in \calY} R_Y, \  \text{and}
\end{equation*} 
\item further, for each $Y \in \calY$ we have $R_Y \cong F_{\infty}^{b_Y} \times \Z^{a_Y}$.	
\end{enumerate}

\ip{Symmetrization} Let $G$, $\X$, and $\F$ be as above.  We explain here how to derive a family of mapping classes from $\F$ that is symmetric in the sense that the subsets of elements supported on two components of $X \in \X$ in the same $\Stab_G(X)$--orbit are conjugate in $\Stab_G(X)$.  

First we require some notation.  Fix some $X \in \X$.  Given a component $X' \in \pi_{0}(X)$, let $\F_{X'}$ be the subset of $\sigma^{-1}(X)$ consisting of all elements whose support is $X'$.  The set $\sigma^{-1}(X)$ is the disjoint union of $\F_{X'}$ over the components $X' \in \pi_{0}(X)$.  

Suppose now that $X_0,\ldots,X_k \in \pi_0(X)$ lie in the same $\Stab_G(X)$--orbit.  Let $g_0 \in G$ be the identity and fix elements $g_1,\ldots,g_k \in \Stab_{G}(X)$ such that $g_i X_0 = X_i$.  Then consider the set
\[
\bigcup_{i,j} \, (g_jg_i^{-1})\F_{X_i}(g_jg_i^{-1})^{-1}
\]
and let $\widehat \F$ be the union of these sets over all $X \in \X$ and all $\Stab_G(X)$--orbits of components of $X$.  We refer to $\widehat \F$ as the \emph{symmetrization} of $\F$.   

There is an induced function $\hat \sigma$ from $\widehat \F$ to the set of components of elements of $\X$; this function takes an element of $\widehat \F$ to its support.  

For $X \in \X$ the normal closure in $\Stab_G(X)$ of the union of the sets $\hat \sigma^{-1}(X')$ with $X'$ a component of $X$ is equal to the group $R_X$, since the new elements are conjugates of the originals by elements in $\Stab_{G}(X)$.  The sets $\hat \sigma^{-1}(X')$ are $G$--independent and each element of $\widehat \F$ is NEC.   

\ip{Applying Theorem~\ref{thm:base} when the surface has boundary.}  To understand the group structure of $R_X$ for $X \in \X$, we would like to apply Theorem~\ref{thm:base} to the action of $\hat \sigma^{-1}(X') \subset \Stab_{G}(X')$ on $\C(X')$ for each non-annular component $X'$ of $X$.  Unfortunately the elements in $\hat \sigma^{-1}(X')$ are not necessarily WPD elements for this action.   Indeed, if $c$ is a component of $\partial X'$ then $T_c$ acts trivially on $\C(X')$ and so if a power of $T_c$ lies in $G$ then no element of $\Stab_{G}(X')$ is WPD.  

Instead, we proceed as in the discussion on partial pseudo-Anosov mapping classes from the introduction.  There is a homomorphism $\Stab_{G}(X') \to \Mod(\bar{X}')$ obtained by collapsing each component of the boundary to a marked point; let $\bar{G}$ denote the image.  As in the proof of Theorem~\ref{thm:base}, there is a natural projection complex that $\bar{G}$ acts on.  The vertices correspond to the quasi-axes bundles in $\C(\bar{X}')$ for the images of the elements in $\hat \sigma^{-1}(X')$.  

Using the action of $\Stab_{G}(X')$ on this projection complex, we can argue as in the proof of Theorem~\ref{thm:base} that for any $t \geq 0$ there is an $N$ so that for $n \geq N$ the subgroup
\[
\normalclosure{ f^{n} \mid f \in \hat\sigma^{-1}(X')}_{\Stab_{G}(X')}
\]
is an infinitely generated free group with basis consisting of conjugates of the various $f^{n}$ and that the translation length of any nontrivial element in this subgroup on $\C(X')$ is at least $t$.

\begin{proof}[Proof of Theorem~\ref{thm:main}]

As in the statement, let $\X$ be a $G$--overlapping family of subsurfaces and let $\F \subset G$ be a finite family of mapping classes that are carried by $\X$ and are $G$--independent.  

If $\X = \{S\}$, then each element of $\F$ is pseudo-Anosov.  In particular, $\F$ is a collection of $G$--independent NEC pseudo-Anosov mapping classes.  The theorem then follows from Theorem~\ref{thm:base}, applying the above discussion in the case that $\partial S \neq \emptyset$. 

We may thus assume in the remainder that $\X \neq \{S\}$.  Let $\Y = G \cdot \X$ and let $\P$ be the projection complex obtained by Proposition~\ref{prop:disconnected subsurfaces}.  Recall that the distance functions are
\begin{equation*}\label{eq:projection sum}
\dd_{Y}(X,Z) = \sum_{Y' \in \pi_{0}(Y)} d_{Y'}(X,Z),
\end{equation*}
where $d_{Y'}$ is the subsurface projection distance.  Let $L = L(\P)$ be the constant from Theorem~\ref{thm:spin}.    

Let $\widehat \F$ be the symmetrization of $\F$ with corresponding function $\hat \sigma$, as defined at the start of the section.  Fix some $X \in \X$.  For $X' \in \pi_0(X)$ we define
\[
\widehat R_{X'} = \normalclosure{ f^n \mid f \in \hat\sigma^{-1}(X')}_{\Stab_G(X')}.
\]
The proof proceeds in the following three steps.
\begin{enumerate}
\item There exists an $N_X$ so that for $X' \in \pi_0(X)$ and $n \geq N_X$ the following statements hold:
\begin{enumerate}
    \item $\widehat R_{X'}$ is $\Z$ if $X'$ is an annulus and $F_\infty$ otherwise, and
    \item the translation length of each element of $\widehat R_{X'}$  acting on $\C(X')$ is at least $L$.
\end{enumerate}
\item $R_{X} \cong F_{\infty}^{b_X} \times \Z^{a_X}$.
\item The action of $R_{X}$ on $\P$ is $L$--spinning.
\end{enumerate}
Since $\X$ is finite, the theorem follows from these statements and Theorem~\ref{thm:spin}, taking $N$ to be the maximum of the $N_X$.

\medskip \noindent {\it Step 1.}  We will define a constant $N_{X'}$ for each component $X' \in \pi_{0}(X)$ and take $N_{X}$ to be maximum of these.

If $X'$ is an annulus, then $\hat \sigma^{-1}(X')$ consists of a single element $f$ which is a power of the Dehn twist supported on this annulus.  We have  $d_{X'}(x,f^{n}x) \geq \abs{n} + 2$ for any $x \in \C(S)$ that intersects $X'$.  As $\grp{f}$ is normal in $\Stab_{G}(X')$ we have that $\widehat R_{X'} = \grp{f^{n}} \cong \Z$.  In this case we define $N_{X'}$ to be~$L$.

It remains to treat the case where $X'$ is non-annular.  As explained at the start of the section, the proof of Theorem~\ref{thm:base} shows that there is a constant $N_{X'}$ such that for $n \geq N_{X'}$ the subgroup $\widehat R_{X'}$ is an infinitely generated free group with basis consisting of conjugates of the various $f^{n}$ and such that the translation length of any nontrivial element of $\widehat R_{X'}$ on $\C(X')$ is at least $L$.  

\medskip \noindent {\it Step 2.}  As $X$ is $G$--overlapping, it follows that for each $X' \in \pi_{0}(X)$ the group $\Stab_{G}(X')$ is a subgroup of $\Stab_{G}(X)$.  There is a natural function
\[
\Psi \from \prod_{X' \in \pi_{0}(X)} \widehat R_{X'} \to \Stab_{G}(X)
\]
that multiplies the coordinate entries of an element in the abstract direct product.  The function $\Psi$ is a well-defined homomorphism because elements in the factor subgroups of $\prod_{X' \in \pi_{0}(X)} \widehat R_{X'}$ have support in a single component of $X$ and elements with disjoint support in $\Mod(S)$ commute.

By Step 1(a), we have that $\prod_{X' \in \pi_0(X)} \widehat R_{X'} \cong F_{\infty}^{b_{X}} \times \Z^{a_{X}}$.  It follows from the definition of symmetrization that the image of $\Psi$ is $R_X$.  It remains to show that $\Psi$ is injective.

Suppose $\Psi(f)$ is trivial and that $f$ is nontrivial. It follows that each coordinate $f_{X'}$ of $f$ is a product of powers of Dehn twists about components of $\partial X'$.  Further, there must be at least one nonannular component $X' \in \pi_{0}(X)$ where $f_{X'}$ is nontrivial, as $X$ contains no parallel annuli.  As each nontrivial element of $\widehat R_{X'}$ has positive translation length on $\C(X')$ by Step 1(b), and since Dehn twists about components of $\partial X'$ act trivially on $\C(X')$, this is a contradiction. 

\medskip 

\noindent {\it Step 3.}  Suppose that $f \in R_{X}$ is nontrivial and $Y \in \Y - \{X\}$.  We write $f$ as a product of elements $f_{X'}$ where $f_{X'} \in \widehat R_{X'}$.  Then we have
\begin{align*}
\dd_{X}(Y,fY) & = \sum_{X' \in \pi_{0}(X)} d_{X'}(Y,f_{X'}Y) \geq L,
\end{align*}
since at least one of the $f_{X'}$ is nontrivial and hence by Step 1 its translation length in $\C(X')$ is at least $L$.  This completes the proof.
\end{proof}


\section{Applications}\label{sec:applications}

In this section we use Theorem~\ref{thm:main} to prove two theorems.  First we prove  Theorem~\ref{thm:cong}, which gives an explicit construction of a normal subgroup of $\Mod(S_g)$ that is not contained in any proper level $m$ congruence subgroup.  We then prove Theorem~\ref{thm:evenodd}, which gives an explicit example of a pseudo-Anosov mapping class $f \in \Mod(S_g)$ with the property that all nonzero even powers of $f$ normally generate a free subgroup of infinite rank and all odd powers of $f$ normally generate $\Mod(S_g)$.  

\subsection{Thurston's construction}\label{sec:thurston} Our examples will be produced from the Thurston construction of pseudo-Anosov mapping classes.  We begin by recalling a special case of Thurston's construction, which works for Dehn twists, and then explain the general case.

Let $c$ and $d$ be curves in $S_g$ that lie in minimal position and have positive geometric intersection number $i(c,d)$.  Let $X$ be the surface with marked points obtained as follows: we take a closed regular neighborhood of $c \cup d$ in $S_g$, take the union of this neighborhood with any complementary regions in $S_g$ that are disks, and then collapse each remaining component of the boundary to a marked point.  By construction, the curves $c$ and $d$ fill $X$, meaning that the complementary regions are disks with at most one marked point.  If $c$ and $d$ fill $S_g$ to begin with, then $X=S_g$.

The curves $c$ and $d$ induce a cell decomposition of $X$ (the vertices are the points of $c \cap d$) and the dual complex is a square complex.  There is a singular Euclidean structure on $X$ where each 2-cell of the square complex is a Euclidean square.  In what follows we take $X$ to be endowed with this structure.

There are two transverse measured foliations on $X$ where the leaves are geodesics parallel to $c$ and $d$, respectively, and where the transverse measures are given by the metric.  We refer to these foliations as the horizontal and vertical foliations of $X$.  These foliations have singularities at the vertices of the square complex.

Let $\Aff(X)$ denote the group of orientation-preserving homeomorphisms of $X$ that preserve the affine structure on $X$ induced by the singular Euclidean metric associated to $X$.  If the complement of the marked points in $X$ has negative Euler characteristic, then homotopic elements of $\Aff(X)$ are equal, and so we may regard $\Aff(X)$ as a subgroup of $\Mod(X)$.  We denote by $\Isom(X)$ the finite subgroup of $\Aff(X)$ consisting of isometries of the singular Euclidean metric \cite[Expose 9]{flp}.  

The horizontal and vertical foliations on $X$ give an orthonormal frame field, well defined up to sign: at each point, we take unit tangent vectors pointing in the horizontal and vertical directions so that the two vectors (in that order) agree with some fixed orientation on $X$.   With these coordinates we obtain a derivative map
\[
D \from \Aff(X) \to \PSL_{2}\R
\]

The Dehn twists $T_c$ and $T_d$ both lie in $\Aff(X)$.  Taking $n=i(c,d)$ we have
\[
DT_c = \left( {\begin{array}{rr}
   1 &  n\\
   0 & 1 \\
  \end{array} } \right)\quad \text{and} \quad
  DT_d = \left( {\begin{array}{rr}
   1 &  0\\
   -n & 1 \\
  \end{array} } \right).
\]

Thurston \cite[Theorem 7]{thurston} proved that $f \in \langle T_c , T_d \rangle$ is periodic, reducible, or pseudo-Anosov exactly according to whether $Df$ is elliptic, parabolic, or hyperbolic.  In the pseudo-Anosov case, the stretch factor of $f$ is equal to the absolute value of the leading eigenvalue of $Df$.  

The singular Euclidean metric on $X$ associated to the stable and unstable foliations of a pseudo-Anosov $f \in \langle T_c , T_d \rangle$ is equal to the original one defined in terms of the curves $c$ and $d$.  In particular, the stable and unstable foliations are geodesic, and we refer to the corresponding directions in the tangent spaces of the nonsingular points as the stable and unstable directions.  Because the two singular Euclidean structures coincide, the elementary closure of $f$ is a subgroup of $\Aff(X)$.

A multicurve in $S_g$ is a collection of pairwise disjoint curves in $S_g$.  Given a multicurve $A$, the associated multi-twist is the product of the Dehn twists about the curves in $A$.  Given two multicurves $A$ and $B$ in $S_g$, there is an analogous Thurston construction.  We may form the surface $X$ as in the case where $A$ and $B$ are curves.  Instead of a square decomposition we use a rectangle decomposition; the lengths and widths of the rectangles are completely determined by the pairwise intersection numbers of the curves in $A$ and $B$.  As before, there is an associated flat structure on $X$ and a homomorphism $D \from \Aff(X) \to \PSL_{2}\R$.  The Nielsen--Thurston type of $f$ is determined by $Df$ in the same way as before.  

\subsection{Two lemmas} We now state and prove the two lemmas (and a corollary) that will be used to prove Theorems~\ref{thm:cong} and~\ref{thm:evenodd}.  The proof of the first lemma is due to Marissa Loving \cite{loving}.

\begin{lemma}
\label{lem:asym curves}
Let $g \geq 2$.  Let $c$ and $d$ be curves in $S_g$ with $i(c,d) > 0$.  Then the mapping classes $f_1 = T_cT_d^{-1}$ and $f_2 = T_cT_d^{-2}$ are normally independent (partial) pseduo-Anosov mapping classes.
\end{lemma}

\begin{proof}

Let $X$ be the singular Euclidean surface with marked points obtained from $c$ and $d$ as above and denote $i(c,d)$ by $n$.  We regard $f_1$ and $f_2$ as elements of $\Mod(X)$.  The mapping classes $f_1$ and $f_2$ fall under the Thurston construction.  Using the derivative map as above, we see that $f_1$ and $f_2$ are pseudo-Anosov with stretch factors 
\[
\lambda_1 = \frac{n^2+2+n\sqrt{n^2+4}}{2} \quad \text{and} \quad \lambda_2 = \quad n^2+1+n\sqrt{n^2+2}.
\]
In particular $\lambda_1$ and $\lambda_2$ lie in the quadratic fields $\mathbb{Q}(\sqrt{m+2})$ and $\mathbb{Q}(\sqrt{m})$, where $m = n^2+2$.  

We claim that $\mathbb{Q}(\sqrt{m+2}) \neq \mathbb{Q}(\sqrt{m})$, which is the same as saying that $m$ and $m+2$ have different square-free parts.  If $m$ is even, then $m+2$ is also even, but exactly one of $m$ and $m+2$ is divisible by 4, and the claim follows.  If $m$ is odd, then no prime factor of $m$ also divides $m+2$, and so again the claim follows.

For $p \in \Z$, the stretch factor of any conjugate of $f_i^p$ is $\lambda_i^{|p|}$.  The stretch factor of a (partial) pseudo-Anosov mapping class is irrational, and so $\lambda_i^p$ is irrational.  If a conjugate of $f_1^{p}$ were equal to a power of $f_2$ then $\lambda_1^{|p|}$ would be an irrational element of both $\mathbb{Q}(\sqrt{m+2})$ and $\mathbb{Q}(\sqrt{m})$.  But two quadratic fields with a common irrational element are equal.  The lemma follows.
\end{proof}

We require one further definition for the statement of the next lemma.  Let $f$ be a pseudo-Anosov mapping class with fixed points $F_+$ and $F_-$ in $\PMF(S_g)$.  Recall that the group $\EC(f)$ is the stabilizer of $\{F_+,F_-\}$ in $\PMF(S_g)$, and $\EC^*(f)$ denotes the subgroup fixing both $F_+$ and $F_-$.  The index of $\EC^*(f)$ in $\EC(f)$ is at most 2.

The following lemma was communicated to us by Chris Leininger \cite{leininger}.  

\begin{lemma}
\label{lem:aff}
Let $A$ and $B$ be multicurves that fill $S_g$, and let $X$ be the associated singular Euclidean surface.  Let $f$ be a pseudo-Anosov element of $\langle T_A, T_B \rangle$, and let $h$ be a periodic element of $\EC(f)$.  Then $h$ preserves $A \cup B$.  Moreover, $h$ lies in $\EC^*(f)$ if and only if $h$ preserves both $A$ and $B$.
\end{lemma}

\begin{proof}
Since $h$ is periodic, $Dh$ is a rotation.  Moreover, this rotation must preserve the pair of eigenspaces for $Df$.  If $Dh$ preserves the two eigenspaces, then $Dh$ is trivial, in which case $h$ fixes the horizontal and vertical directions, hence $A$ and $B$.  If $Dh$ interchanges the eigenspaces, then it must be that the stable and unstable directions are orthogonal and that $Dh$ is rotation by $\pi/2$.  Thus $Dh$ also interchanges the horizontal and vertical directions in $S_g$.  The curves of $A$ and $B$ are exactly the curves with vertical or horizontal trajectories, so $A$ and $B$ are interchanged.
\end{proof}

\begin{lemma}
\label{lem:thurston}
Let $A$ and $B$ be multicurves in $S_g$ that fill $S_g$, and assume that there is no element of $\Mod(S_g)$ interchanging $A$ and $B$.  Let $f$ be a pseudo-Anosov element of $\langle T_A, T_B \rangle$.  Then $f$ is central in $\EC(f)$.  In particular, $f$ is NEC.
\end{lemma}

\begin{proof}

We first claim that $\EC(f) = \EC^*(f)$.  Let $h \in \EC(f)$.  If $h$ has infinite order, then $h^2$ is of infinite order and preserves the points of $\PMF(S_g)$ corresponding to the stable and unstable foliations for $f$.  It follows that $h^2$ is pseudo-Anosov \cite[Expos\'e 9, Lemme 15]{flp}.  By the Nielsen--Thurston classification theorem, $h$ is also pseudo-Anosov.  Since a pseudo-Anosov mapping class has exactly two fixed points in $\PMF(S_g)$, it follows that  $h^2$ and $h$ have the same pair of fixed points, hence $h \in \EC^*(f)$.  If $h$ is of finite order, then it follows from Lemma~\ref{lem:aff} and the assumption on $A$ and $B$ that $h \in \EC^*(f)$.  This completes the proof of the claim.

Let $f_0$ be a root of $f$ with minimal stretch factor.  There is an internal semidirect product decomposition as follows:
\[
\EC^*(f) \cong \langle f_0 \rangle \ltimes \Isom(X).
\]
The proof of this statement can be found in the unpublished paper by McCarthy \cite{mccarthy}.  Clearly $f_0$ commutes with $f$.  Also, since each element of $\Isom(X)$ is of finite order, it follows from Lemma~\ref{lem:aff} and the assumption on $A$ and $B$ that each element of $\Isom(X)$ preserves both $A$ and $B$, and hence commutes with $f$.  The lemma follows.
\end{proof}

\subsection{Proofs of the theorems} We are ready now to prove Theorems~\ref{thm:cong} and~\ref{thm:evenodd}.  The construction in the proof of Theorem~\ref{thm:cong} was suggested by Mladen Bestvina \cite{mladen}.

\begin{proof}[Proof of Theorem~\ref{thm:cong}]

Choose a nonseparating curve $c$ and a separating curve $d$ so that $c$ and $d$  fill the surface $S_g$.  Let $f_1= T_cT_d^{-1}$ and let $f_2 = T_cT_d^{-2}$.  By the Thurston construction, $f_1$ and $f_2$ are pseudo-Anosov,  by Lemma~\ref{lem:asym curves} they are normally independent, and by Lemma~\ref{lem:thurston} they are both NEC.  We may thus apply Theorem~\ref{thm:base}.  Specifically, we may choose distinct prime numbers $p_1$ and $p_2$ so that the normal closure $N$ of $f_1^{p_1}$ and $f_2^{p_2}$ is a free group.  In particular, $N$ is a proper subgroup of $\Mod(S_g)$.  

It remains to show that $N$ is not contained in any congruence subgroup of $\Mod(S_g)$.  Since $T_d$ acts trivially on $H_1(S_g;\Z)$ the action of $f_1^{p_1}$ on $H_1(S_g;\Z)$ is the same as that of $T_c^{p_1}$.  In particular, $f_1^{p_1}$ lies in $\Mod(S_g)[m]$ if and only if $m$ divides $p_1$.  Similarly, $f_2^{p_2}$ lies in $\Mod(S_g)[m]$ if and only if $m$ divides $p_2$.  Thus there is no proper subgroup $\Mod(S_g)[m]$ containing $N$, as desired.
\end{proof}

We were informed by Ashot Minasyan \cite{ashot} of an alternative example a subgroup of $\Mod(S_g)$ satisfying the conclusion of Theorem~\ref{thm:cong}.  This example is in fact not contained in any proper normal subgroup of finite index.  The construction uses a result of Michael Hull, and proceeds as follows.  Let $A$ be a finitely generated group that has no finite quotients other than the trivial group and set $G = A \ast A$.  Notice that $G$ also does not have any finite quotients other than the trivial group.  The action of $G$ on the corresponding Bass--Serre tree shows that $G$ is acylindrically hyperbolic.  Since $\Mod(S_{g})$ is also acylindrically hyperbolic~\cite{ar:Bowditch08}, it follows from a result of Hull that there is a group $Q$ that is a quotient of both $G$ and $\Mod(S)$~\cite[Corollary~1.6]{ar:Hull16}.  As $Q$ is a quotient of $G$, it too does not have any finite quotients other than the trivial group.  Let $K$ be the kernel of the map $\Mod(S_{g}) \to Q$.  If $K$ is contained in a proper normal subgroup $H$ of finite index in $\Mod(S_{g})$ then the image of $H$ in $Q$ is a proper normal finite index subgroup of $Q$, which is a contradiction as $Q$ does not have any finite quotients other than the trivial group.

\begin{figure}
\includegraphics[scale=.35]{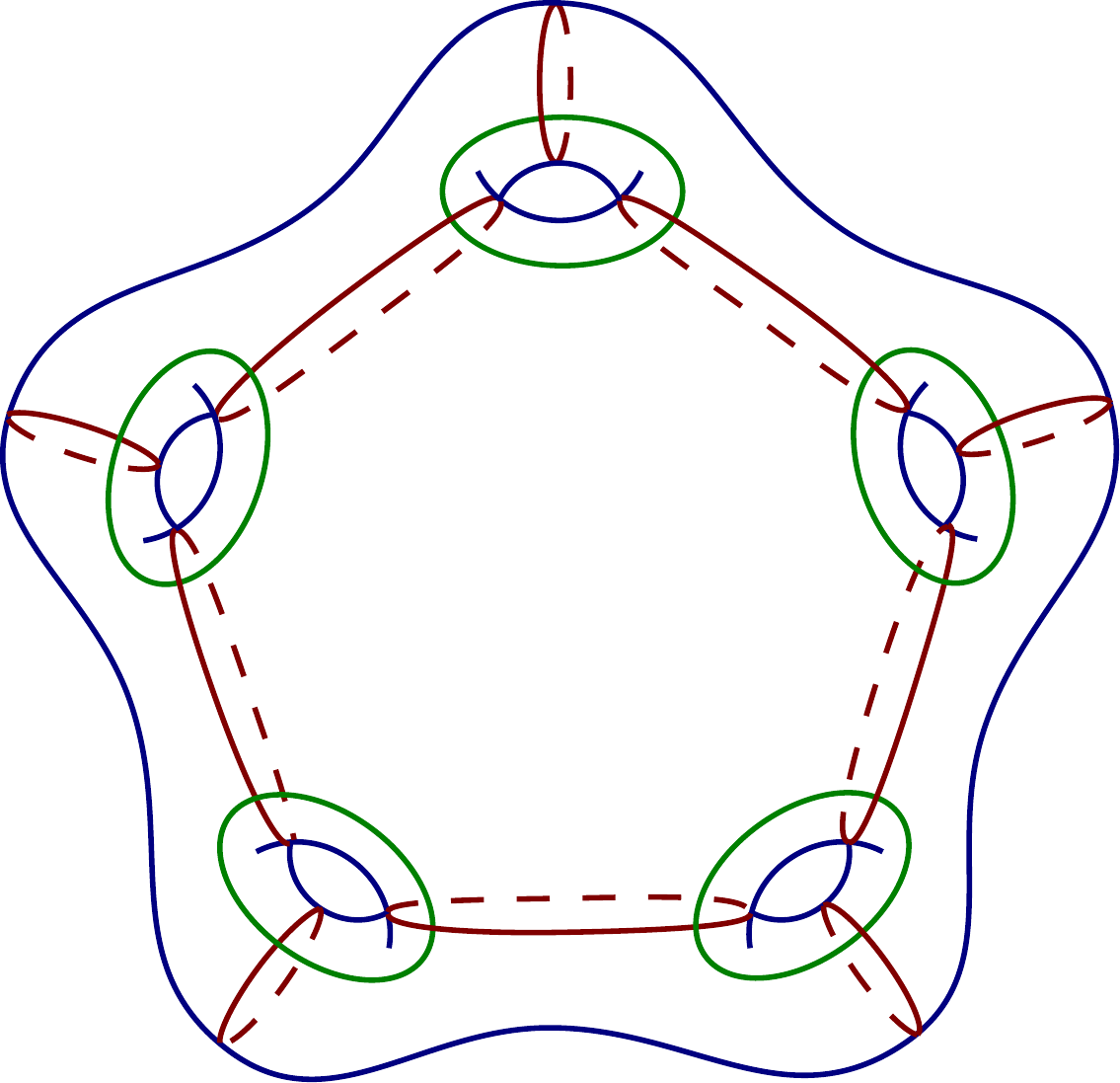}
\caption{The two multicurves used in the proof of Theorem~\ref{thm:evenodd}}
\label{fig:penner}
\end{figure}

\bigskip

The proof of Theorem~\ref{thm:evenodd} is inspired by the work of the third author with Lanier  \cite[Theorem 1.4]{lm}.  Specifically, they gave a recipe for constructing a pseudo-Anosov mapping class with the property that all of its odd powers are normal generators for $\Mod(S_g)$.  Our construction here is an explicit special case of their recipe, designed so that its even powers have the desired property.

\begin{proof}[Proof of Theorem~\ref{thm:evenodd}]

Fix $g \geq 3$.  Let $A$ and $B$ be the multicurves in $S_g$ indicated in Figure~\ref{fig:penner} (there is only one way to partition the set of curves in the figure into two multicurves).  By the Thurston construction the mapping class $f = T_AT_B^{-1}$ is pseudo-Anosov.

Let $D_{2g}$ denote the dihedral group of order $2g$.  There is a standard action of $D_{2g}$ on $S_g$ by orientation-preserving homeomorphisms, giving rise to a subgroup of $\Mod(S_g)$ isomorphic to $D_{2g}$.   We refer to this subgroup as simply $D_{2g}$.  Since this action preserves $A$ and $B$, it follows that $D_{2g}$ lies in the normalizer (indeed, centralizer) of $\langle f \rangle$.  In particular, $D_{2g}$ lies in $\EC(f)$.

Let $k$ be an element of $D_{2g}$ corresponding to a reflection of a $2g$-gon and let $h=kf$.  We will show that some power of $h$ satisfies the conclusion of the theorem.

First, we claim that all odd powers of $h$ have normal closure equal to $\Mod(S_g)$.  So let $n$ be odd.  We follow here the argument of Lanier and the third author  \cite[proof of Theorem 1.4]{lm}.  Since $k$ has order 2 and since $k$ commutes with $f$, we have that $h^n = kf^n$.  Let $r$ denote one of the generators for the cyclic subgroup of $D_{2g}$ of order $g$.  The commutator $[r,h^n] = (rh^nr^{-1})h^{-n}$ lies in the normal closure of $h^n$.  Since $D_{2g}$ lies in the centralizer of $f$ we have that 
\[
[r,h^n] = rh^nr^{-1}h^{-n} = rkf^nr^{-1}kf^{-n} = rkr^{-1}k=r^2,
\]
where the last equality uses the relation $kr^{-1}k=r$ in $D_{2g}$.  Lanier and the third author \cite[Theorem 1.1]{lm} showed that for $g \geq 3$ the normal closure of any nontrivial periodic element of $\Mod(S_g)$ besides a hyperelliptic involution is $\Mod(S_g)$.  Since $r^2$ is nontrivial and is not a hyperelliptic involution (consider, for instance, the action on $H_1(S_g;\Z)$), it follows that the normal closure of $r^2$, hence $h^n$, is $\Mod(S_g)$, as desired.

Next, we claim that all sufficiently large even powers of $h$ have normal closure isomorphic to $F_\infty$.  If $n$ is even then $h^n=(f^2)^n$; so a large even power of $h$ is a large power of $f$.  There is no element of $\Mod(S_g)$ interchanging $A$ and $B$, since these two sets contain different numbers of curves.  Thus by Lemma~\ref{lem:thurston}, the mapping class $f$ is NEC.  By our Theorem~\ref{thm:base}, all sufficiently large powers of $f$, hence all sufficiently large even powers of $h$, have normal closure isomorphic to $F_\infty$, as desired.

If $n$ is odd and sufficiently large, then it follows from the previous two claims that $h^n$ satisfies the conclusion of the theorem.
\end{proof}


\bibliography{Windmills}
\bibliographystyle{acm}

\end{document}